\documentclass[11pt,twoside]{article}

\usepackage[english]{babel}
\usepackage[T1]{fontenc}
\usepackage{lmodern}
\usepackage{amsmath,amsthm,amssymb,amsfonts}
\usepackage[sort&compress,numbers]{natbib}
\usepackage{hyperref}
\usepackage[capitalise]{cleveref}
\usepackage[left=2.5cm,right=2.5cm,top=3cm,bottom=3cm,includeheadfoot]{geometry}
\usepackage{bbm}
\usepackage{xcolor}
\usepackage{enumerate,enumitem}
\usepackage{fancyhdr}

\hypersetup{colorlinks=true,
	linkcolor=teal,
	citecolor=orange}

\theoremstyle{plain}
\newtheorem{thmintro}{Theorem}

\newtheorem*{corintro*}{Corollary}

\newtheorem{thmA}{Theorem}
\newtheorem{thmB}{Theorem}

\theoremstyle{plain}
\newtheorem{lem}{Lemma}[section]
\newtheorem{prop}[lem]{Proposition}
\newtheorem{thm}[lem]{Theorem}
\newtheorem{cor}[lem]{Corollary}

\theoremstyle{definition}
\newtheorem{defi}[lem]{Definition}
\newtheorem{exl}[lem]{Example}
\newtheorem{rem}[lem]{Remark}

\numberwithin{equation}{section}

\crefname{equation}{}{}

\newcommand{\R}{\mathbb{R}}
\renewcommand{\S}{\mathbb{S}}
\newcommand{\MVal}{\mathbf{MVal}}
\newcommand{\K}{\mathcal{K}}
\newcommand{\SCap}{\mathrm{C}^{\S}}
\newcommand{\MSO}{\mathrm{M}}
\newcommand{\M}{\mathcal{M}}
\newcommand{\T}{\mathrm{T}}
\newcommand{\Id}{\mathrm{Id}}
\newcommand{\J}{\mathsf{J}}
\newcommand{\e}{\bar{e}}
\newcommand{\abs}[1]{\lvert #1 \rvert}
\newcommand{\norm}[1]{\lVert #1 \rVert}
\newcommand{\D}{\mathcal{D}}
\newcommand{\A}{\mathcal{A}}
\newcommand{\SO}{\mathrm{SO}}
\newcommand{\TV}{\mathrm{TV}}

\DeclareMathOperator{\tr}{\mathrm{tr}}
\DeclareMathOperator{\Div}{\mathrm{div}}
\DeclareMathOperator{\supp}{supp}
\DeclareMathOperator{\diam}{diam}
\DeclareMathOperator{\cl}{cl}

\begin{document}

\thispagestyle{empty}	
\begin{center}
	\LARGE{\textbf{Fixed Points of Mean Section Operators}}
	
	\vspace{5mm}
	
	\Large{Leo Brauner and Oscar Ortega-Moreno}
\end{center}

\vspace{5mm}

\begin{abstract}
	We characterize rotation equivariant bounded linear operators from $C(\S^{n-1})$ to $C^2(\S^{n-1})$ by the mass distribution of the spherical Laplacian of their kernel function on small polar caps.
	Using this characterization, we show that every continuous, homogeneous, translation invariant, and rotation equivariant Minkowski valuation $\Phi$ that is weakly monotone maps the space of convex bodies with a $C^2$ support function into itself.
	As an application, we prove that if $\Phi$ is in addition even or a mean section operator, then Euclidean balls are its only fixed points in some $C^2$ neighborhood of the unit ball. Our approach unifies and extends previous results by Ivaki from 2017 and the second author together with Schuster from 2021.
\end{abstract}

\section{Introduction}
\label{sec:intro}

Sections and projections of convex bodies play an essential role in the field of geometric tomography.
By taking measurements in lower dimensions, one seeks to recover information about the geometry of the original object.
A common procedure to work with these measurements is to assemble them into a new convex body.
For instance, the projection body of a convex body $K$ is built from the volumes of shadows of $K$ cast from every possible direction. To give the exact definition, recall that a convex body $K$ (that is, a convex, compact subset) in $\R^n$, where throughout $n\geq 3$, can be defined by its support function $h(K,u)=\max\{\langle x,u\rangle: x\in K\}$, $u\in\S^{n-1}$. The \emph{projection body} $\Pi K$ of a convex body $K$ is defined by
\begin{equation*}
	h(\Pi K, u)
	= V_{n-1}(K | u^\perp),
	\qquad u\in\S^{n-1},
\end{equation*}
where $K | u^\perp$ denotes the orthogonal projection of $K$ onto the hyperplane $u^\perp$ and $V_{n-1}$ denotes the $(n-1)$-dimensional volume.
The geometric operator $\Pi$ was already introduced by Minkowski and has since become central to convex geometry (see, e.g., \cite{MR0362057,MR1119653,MR1031458,MR1035998,MR301637,MR1942402,MR4508339,MR2563216,MR1863023,MR1042624}).

A more recent instance of the procedure described above is the family of \emph{mean section operators} introduced by Goodey and Weil \cite{MR1171176,MR2854184}. 
For $0\leq j\leq n$, the $j$-th mean section body $\MSO_j K$ of a convex body $K$ is essentially the average of all $j$-dimensional sections of $K$ with respect to Minkowski addition. More precisely,
\begin{equation*}
	h(\MSO_j K, u)
	= \int_{\mathrm{AG}(n,j)} h(K\cap E , u) dE,
	\qquad u \in\S^{n-1},
\end{equation*}
where $\mathrm{AG}(n,j)$ denotes the affine Grassmannian (that is, the space of $j$-dimensional affine subspaces of $\R^n$) and integration is with respect to a suitably normalized, positive, rigid motion invariant measure.

The projection body and mean section operators belong to a rich class of geometric operators acting on the space $\K^n$ of convex bodies in $\R^n$. A \emph{Minkowski valuation} is a map $\Phi:\K^n\to\K^n$ satisfying
\begin{equation*}
	\Phi K + \Phi L
	= \Phi(K\cup L) + \Phi(K\cap L)
\end{equation*}
with respect to Minkowski addition whenever $K\cup L \in\K^n$. Scalar valued valuations have a long history in convex geometry (see, e.g., \cite{MR1709308,MR1837364,MR3820854,MR2776365,MR3194492,MR1608265,MR2680490,MR3716941}). Their systematic study goes back to Hadwiger's \cite{MR0102775} famous characterization of the intrinsic volumes $V_i$, $0\leq i\leq n$, (see \cref{sec:background}) as a basis for the space of continuous, rigid motion invariant scalar valuations.

The investigation of Minkowski valuations has originated from Schneider's~\cite{MR353147} research on Minkowski endomorphisms. However, it was the seminal work by Ludwig~\cite{MR1942402,MR2159706} which prompted further development. In \cite{MR1942402}, Ludwig identifies Minkowski's projection body map as the unique (up to a positive constant) continuous, translation invariant, affine contravariant Minkowski valuation, solving a problem posed by Lutwak. Following Ludwig's steps, contributions of several authors (e.g., \cite{MR2793024,MR3880207,MR3715395,MR2966660,MR2772547,MR2846354,MR2997003}) show that the convex cone of Minkowski valuations compatible with affine transformations is in many instances finitely generated. In contrast, a less restrictive condition such as rotation equivariance produces a significantly larger class of valuations, making their classification challenging.

Denote by $\MVal$ the space of all continuous, translation invariant Minkowski valuations \linebreak intertwining rotations and by $\MVal_i$ the subspace of Minkowksi valuations homogeneous of degree~$i$. A map $\Phi:\K^n\to\K^n$ is said to have \emph{degree} $i$ if $\Phi(\lambda K)=\lambda^i \Phi K$ for every $K\in\K^n$ and $\lambda\geq 0$. By a classical result of McMullen \cite{MR448239}, continuous, translation invariant, homogeneous valuations can only have integer degree $i\in\{0,\ldots,n\}$. 
In recent years, substantial progress (e.g., \cite{MR2238926,MR2327043,MR2668553,MR3432270,MR3854893}) to obtain a Hadwiger-type theorem for the space $\MVal$ has led to the following representation by Dorrek~\cite{MR3655954} involving the \emph{spherical convolution} of an integrable function and the \emph{area measures} $S_i(K,\cdot)$ of a convex body $K$ (see \cref{sec:background}): for every $\Phi_i\in\MVal_i$ of degree $1\leq i\leq n-1$, there exists a unique centered, $\SO(n-1)$ invariant function $f\in L^1(\S^{n-1})$ such that for every $K\in\K^n$,
\begin{equation}\tag{1}\label{eq:MVal_representation:intro}
	h(\Phi_i K, \cdot)
	= S_i(K,\cdot)\ast f. 
\end{equation}
A function on $\S^{n-1}$ is said to be \emph{centered} if it is orthogonal to all linear functions. 
We call the function $f$ in \cref{eq:MVal_representation:intro} the \emph{generating function} of $\Phi_i$. If $\Phi_iK=\{o\}$ for all $K\in\K^n$, we call $\Phi_i$ \emph{trivial}.

For $1\leq i\leq n-1$, the $i$-th projection body map $\Pi_i$ is defined by $h(\Pi_iK,u)=V_i(K| u^\perp)$, $u\in\S^{n-1}$. Note that $\Pi_{n-1}=\Pi$ is Minkowski's projection body map. Each operator $\Pi_i$ belongs to $\MVal_i$ and is generated by the support function of a line segment, as can be easily deduced from Cauchy's projection formulas (see, e.g., \cite[p.~408]{MR2251886}). The $j$-th mean section operator $\MSO_j$, up to a suitable translation, also belongs to $\MVal_i$, where $i=n-j+1$. However, unlike for the projection body map, the determination of their generating functions is non-trivial and involves the functions employed by Berg in his solution of the Christoffel problem \cite{MR1579295}. For each dimension $n\geq 2$, Berg~\cite{MR254789} constructed a function $g_n\in C^{\infty}(-1,1)$ with the property that
$h(K,\cdot) = S_1(K,\cdot) \ast \breve{g}_n+\langle s(K),\cdot\rangle$	 
for every convex body $K\in\K^n$, where $\breve{g}_n$ is the $\SO(n-1)$ invariant function on $\S^{n-1}$ associated to $g_n$ and $s(K)$ denotes the \emph{Steiner point} of $K$ (see \cref{sec:background}). Goodey and Weil \cite{MR1171176,MR3263512} showed that for $2\leq j\leq n$ and every $K\in\K^n$, 
\begin{equation}\tag{2}\label{eq:MSO_representation}
 	h(\MSO_j K,\cdot)
 	= S_{n-j+1}(K,\cdot)\ast\breve{g}_j + c_{n,j}V_{n-j}(K)\langle s(K),\cdot\rangle,
\end{equation}
where $c_{n,j}>0$ is some constant.

Fixed points of geometric operators are closely related to a range of open problems in convex geometry (see, e.g., \cite{MR84791,MR0362057,MR2251886}). For instance, \emph{Petty's conjecture} \cite{MR0362057} can be expressed in terms of fixed points of $\Pi^2$ up to affine transformations. This was first observed by Schneider \cite{MR880200} and later extended by Lutwak~\cite{MR1042624} to projection body maps of all degrees. The global classification of fixed points of $\Pi_i^2$ has only been settled in the polytopal case by Weil \cite{MR301637} and in the $1$-homogeneous case by Schneider~\cite{MR500538}. It is conjectured that the only smooth fixed points of $\Pi^2$ are ellipsoids. Locally around the unit ball, this was recently confirmed independently by Saraoglu and Zvavitch \cite{MR3571903} and Ivaki \cite{MR3757054}, motivated by the work of Fish, Nazarov, Ryabogin, and Zvavitch \cite{MR2739787}.

For degree $1< i<n-1$, Ivaki \cite{MR3639524} showed that in some $C^2$ neighborhood of the unit ball, the only fixed points of $\Pi_i^2$ are Euclidean balls. The second author and Schuster \cite{MR4316669,OM2023} have shown that this phenomenon also holds for the class of even \emph{$C^2_+$ regular} Minkowski valuations, that is, Minkowski valuations generated by the support function of an origin-symmetric convex body of revolution that has a $C^2$ boundary with positive Gauss curvature. A line segment is clearly not of this kind, so the results by Ivaki and by the second author and Schuster appear to be disconnected. In this paper, we bridge this gap with our first main result.

\begin{thmintro}\label{fixed_points_body_of_rev:intro}
	Let $1< i\leq n-1$ and $\Phi_i\in\MVal_i$ be generated by an origin-symmetric convex body of revolution. Then there exists a $C^2$ neighborhood of the unit ball where the only fixed points of $\Phi_i^2$ are Euclidean balls, unless $\Phi_i$ is a multiple of the projection body map, in which case ellipsoids are also fixed points.
\end{thmintro}

The case when $i=1$ (that is, $\Phi_1\in\MVal_1$ is a Minkowski endomorphism) has been settled globally by Kiderlen~\cite{MR2238926}.
With \cref{fixed_points_body_of_rev:intro}, we unify the previous results on $C^2_+$ regular Minkowski valuations and projection body maps obtained in \cite{MR4316669} and \cite{MR3757054,MR3639524}, respectively. However, none of them (including \cref{fixed_points_body_of_rev:intro}) cover any local uniqueness of fixed points of mean section operators. This is because Berg's functions are neither even nor support functions. By further extending the techniques employed in the proof of \cref{fixed_points_body_of_rev:intro}, we obtain the following.

\begin{thmintro}\label{fixed_points_MSO:intro}
	For $2 \leq j < n$, there exists a $C^2$ neighborhood of the unit ball where the only fixed points of $\MSO_j^2$ are Euclidean balls.
\end{thmintro}

Throughout, this is to be understood as follows: there exists some $\varepsilon>0$ such that if $K\in\K^n$ has a $C^2$ support function satisfying $\norm{h(\alpha K+x,\cdot)-1}_{C^2(\S^{n-1})}<\varepsilon$ for some $\alpha>0$ and $x\in\R^n$, and if $\MSO_j^2K$ is a dilated and translated copy of $K$, then $K$ is a Euclidean ball.
We want to emphasize that we will obtain both Theorem~\ref{fixed_points_body_of_rev:intro} and \ref{fixed_points_MSO:intro} from a more general result (\cref{w_mon=>fixed_points_Phi^2}) that applies to all weakly monotone, homogeneous Minkowski valuations in the space $\MVal$. \linebreak A Minkowski valuation $\Phi:\K^n\to\K^n$ is called \emph{weakly monotone} if $\Phi K \subseteq \Phi L$ whenever $K\subseteq L$ and the Steiner points of $K$ and $L$ are at the origin. The proof of \cref{w_mon=>fixed_points_Phi^2} requires the convolution transform defined by its generating function to be a bounded operator from $C(\S^{n-1})$ to $C^2(\S^{n-1})$. The following theorem provides necessary and sufficient conditions for the boundedness of convolution transforms.

\begin{thmintro}\label{convol_trafo_C^2:intro}
	Let $f\in L^1(\S^{n-1})$ be $\SO(n-1)$ invariant. Then the convolution transform $\phi\mapsto \phi\ast f$ is a bounded linear operator from $C(\S^{n-1})$ to $C^2(\S^{n-1})$ if and only if $\square_nf$ is a signed measure and
	\begin{equation}\tag{3}\label{eq:convol_trafo_C^2:intro}
		\int_{(0,\frac{\pi}{2})} \frac{1}{r}~\abs{(\square_n f)(\{u\in\S^{n-1}:\langle\pm\e,u\rangle>\cos r\})} ~dr
		< \infty.
	\end{equation}
\end{thmintro}

Here, $\e\in\S^{n-1}$ denotes the north pole of the unit sphere fixing the axis of revolution of $f$ and \linebreak $\square_nf=\frac{1}{n-1}\Delta_{\S}f + f$, where $\Delta_{\S}$ is the spherical Laplacian on $\S^{n-1}$. \cref{convol_trafo_C^2:intro} tells us that the regularity of the convolution transform defined by $f$ is determined by the mass distribution of $\square_nf$ on small polar caps.
It turns out that generating functions of weakly monotone Minkowski valuations exhibit precisely this behavior.

\begin{thmintro}\label{w_mon=>w_pos+density:intro}
	Let $1\leq i\leq n-1$ and $\Phi_i\in\MVal_i$ with generating function $f$. Then $f$ is locally Lipschitz outside the poles and $\square_nf$ is a signed measure on $\S^{n-1}$. Moreover, if $\Phi_i$ is in addition weakly monotone, then there exists $C> 0$ such that for all $r\geq 0$,
	\begin{equation}\tag{4}\label{eq:w_mon=>w_pos+density:intro}
		\abs{\square_nf}\big(\{u\in\S^{n-1}: \abs{\langle\e,u\rangle}>\cos r\}\big)
		\leq Cr^{i-1}.
	\end{equation}
\end{thmintro} 

As an immediate consequence of Theorems~\ref{convol_trafo_C^2:intro} and \ref{w_mon=>w_pos+density:intro}, we obtain the following. 

\begin{corintro*}\label{w_mon=>bounded_C^2:intro}
	Let $1<i\leq n-1$ and $\Phi_i\in\MVal_i$ be weakly monotone with generating function $f$. Then the convolution transform $\phi \mapsto \phi\ast f$ is a bounded linear operator from $C(\S^{n-1})$ to $C^2(\S^{n-1})$.
	In particular, $\Phi_i$ maps the space of convex bodies with a $C^2$ support function into itself.
\end{corintro*}

To the best of our knowledge, apart from smooth Minkowski valuations, this was previously only known for the projection body operators: it was shown by Martinez-Maure \cite{MR1814896} that the cosine transform is a bounded linear operator from $C(\S^{n-1})$ to $C^2(\S^{n-1})$, which is an essential tool in the proof of Ivaki's \cite{MR3757054,MR3639524} fixed point results.

We want to remark that the continuity of $f$ proven in \cref{w_mon=>w_pos+density:intro} confirms a conjecture by Dorrek. Moreover, note that \cref{eq:w_mon=>w_pos+density:intro} relates the regularity of $f$ to the degree of homogeneity. It has been shown by Parapatits and Schuster \cite{MR2921168} that if a function generates a Minkowski valuation of a certain degree, then it also generates Minkowski valuations of all lower degrees. Using \cref{eq:w_mon=>w_pos+density:intro}, it can be shown that for each $1\leq i\leq n-1$, the Berg function $g_{n-i+1}$ generates a weakly monotone Minkowski valuation of degree $i$ but not higher. 


The article is organized as follows. In \cref{sec:background}, we collect the required background on convex geometry and analysis on the unit sphere. In \cref{sec:convol}, we investigate regularity of zonal measures and convolution transforms, proving \cref{convol_trafo_C^2:intro}. In \cref{sec:w-mon}, we show that weak monotonicity of Minkowski valuations implies additional regularity of the generating function, proving \cref{w_mon=>w_pos+density:intro}. Finally, in \cref{sec:fixed_points} we apply our results from the previous sections to the study of fixed points. There we prove Theorems~\ref{fixed_points_body_of_rev:intro} and \ref{fixed_points_MSO:intro} as well as a general result on even Minkowski valuations.

\pagebreak

\section{Background Material}
\label{sec:background}

In the following, we collect basic facts about convex bodies, mixed volumes and area measures. We also discuss differential geometry and the theory of distributions on the unit sphere. In the final part of this section, we gather the required material from harmonic analysis, including spherical harmonics and the convolution of measures. As general references for this section, we cite the monographs by Gardner \cite{MR2251886}, Schneider \cite{MR3155183}, Hörmander \cite{MR1065136}, Lee \cite{MR2954043,MR3887684}, and Groemer \cite{MR1412143}.

\paragraph{Convex geometry.}

The space $\K^n$ of convex bodies naturally carries an algebraic and topological structure. The so-called Minkowski operations, dilation and the Minkowski addition, are given by $\lambda K=\{\lambda x:x\in K\}$, $\lambda\geq 0$, and $K+L=\{x+y:x\in K,y \in L\}$. The Hausdorff metric can be defined as
\begin{equation*}
	d(K,L)
	= \max\{t\geq0: K\subseteq L+tB^n\text{ and }L\subseteq K+tB^n \},
\end{equation*}
where $B^n$ denotes the unit ball of $\R^n$.

As was pointed out before, every convex body $K\in\K^n$ is uniquely determined by its support function $h_K(x)=h(K,x)=\max\{\langle x,y\rangle:y\in K\}$, $x\in\R^n$, which is homogeneous of degree one and subadditive. Conversely, every function with these two properties is the support function of a unique body $K\in\K^n$. Associating a convex body with its support function is compatible with the structure of $\K^n$, that is, $h_{\lambda K+L}=\lambda h_K+h_L$ and $d(K,L)=\norm{h_K-h_L}_{\infty}$, where $\norm{\cdot}_\infty$ denotes the maximum norm on the unit sphere. Moreover, $K\subseteq L$ if and only if $h_K\leq h_L$. In addition, $h_{\vartheta K+x}(u)=h_K(\vartheta^{-1}u)+\langle x,u\rangle$ for every $\vartheta\in\SO(n)$ and $x\in\R^n$.

The Steiner formula expresses the volume of the parallel set of a convex body $K$ at distance $t\geq 0$ as a polynomial in $t$. To be precise,
\begin{equation}\label{eq:Steiner_formula}
	V_n(K+tB^n)
	= \sum_{i=0}^n t^{n-i}\kappa_{n-i}V_i(K),
\end{equation}
where $\kappa_i$ denotes the $i$-dimensional volume of $B^i$ and the coefficient $V_i(K)$ is called the $i$-th \emph{intrinsic volume} of $K$ for $0\leq i\leq n$. The intrinsic volumes are important quantities carrying geometric information on convex bodies. For instance, $V_n$ is the volume, $V_{n-1}$ the surface area, and $V_1$ the mean width.

The \emph{surface area measure} $S_{n-1}(K,\cdot)$ of a convex body is the positive measure on $\S^{n-1}$ defined as follows: the measure $S_{n-1}(K,A)$ of a measurable subset $A\subseteq \S^{n-1}$ is the $(n-1)$-dimensional Hausdorff measure of all boundary points of $K$ with outer unit normal in $A$.
Analogously to \cref{eq:Steiner_formula}, there is a Steiner-type formula for surface area measures:
\begin{equation*}
	S_{n-1}(K+tB^n,\cdot)
	= \sum_{i=0}^{n-1} t^{n-1-i}\tbinom{n-1}{i}S_i(K,\cdot),
\end{equation*}
where the measure $S_i(K,\cdot)$ is called the $i$-th \emph{area measure} of $K$ for $0\leq i\leq n-1$. Each of the area measures is \emph{centered}, meaning that they integrate all linear functions to zero. By a theorem of Alexandrov-Fenchel-Jessen (see, e.g., \cite[Section~8.1]{MR3155183}), if $K$ has non-empty interior, then each area measure $S_i(K,\cdot)$ determines $K$ up to translations.

The \emph{Steiner point} of a convex body $K$ is defined as $s(K)= \int_{\S^{n-1}} h(K,u)udu$. The Steiner point map $s:\K^n\to\R^n$ is the unique continuous, vector valued valuation intertwining rigid motions (see, e.g., \cite[p.~181]{MR3155183}).

\paragraph{Differential geometry.}

As an embedded submanifold of $\R^n$, the unit sphere $\S^{n-1}$ naturally inherits the structure of an $(n-1)$-dimensional Riemannian manifold. We identify the tangent space at each point $u\in\S^{n-1}$ with $u^\perp\subseteq\R^n$, which allows us to interpret tensor fields as maps from $\S^{n-1}$ into some Euclidean space.

Throughout, we will only work with tensor fields up to order two. That is, we define a vector field on $\S^{n-1}$ as a map $X:\S^{n-1}\to\R^n$ such that $X(u)\in u^\perp$ for every $u\in\S^{n-1}$, and a $2$-tensor field on $\S^{n-1}$ as a map $Y:\S^{n-1}\to\R^{n\times n}$ such that $Y(u)(u^\perp)\subseteq u^\perp$ and $Y(u)u=0$ for every $u\in\S^{n-1}$. For instance, let $Y(u)=P_{u^\perp}$ be the orthogonal projection onto $u^\perp$ for each $u\in\S^{n-1}$. Then the $2$-tensor field $Y$ acts as the identity on each tangent space.
The inner product of two $2$-tensors $Y_1$ and $Y_2$ on $u^\perp$ is given by $\langle Y_1,Y_2\rangle=\tr(Y_1Y_2)$.

We denote by $\nabla_{\S}$ the standard covariant derivative and by $\Div_{\S}$ the divergence operator on $\S^{n-1}$.
The operators $\nabla_{\S}$ and $\Div_{\S}$ are related via the \emph{spherical divergence theorem}, which states that
\begin{equation*}
	\int_{\S^{n-1}} \langle X(u),\nabla_{\S}\phi(u)\rangle du
	= - \int_{\S^{n-1}} \Div_\S X(u) \phi(u) du
\end{equation*}
and
\begin{equation*}
	\int_{\S^{n-1}} \langle Y(u),\nabla_{\S}X(u) \rangle du
	= - \int_{\S^{n-1}} \langle \Div_{\S} Y(u), X(u) \rangle du
\end{equation*}
for every smooth function $\phi$, smooth vector field $X$, and smooth $2$-tensor field $Y$ on $\S^{n-1}$.

The \emph{spherical gradient} $\nabla_{\S}\phi$ and \emph{spherical Hessian} $\nabla_{\S}^2\phi$ of a smooth function $\phi$ can be expressed in terms of derivatives along smooth curves. If $\gamma:I\to\S^{n-1}$ is a geodesic in $\S^{n-1}$, then
\begin{equation*}
	\frac{d}{ds}\bigg|_0 \phi(\gamma(s))
	= \langle \nabla_{\S}\phi(\gamma(0)), \gamma'(0) \rangle
	\qquad\text{and}\qquad
	\frac{d^2}{ds^2}\bigg|_0 \phi(\gamma(s))
	= \langle \nabla_{\S}^2\phi(\gamma(0)), \gamma'(0)\otimes\gamma'(0) \rangle.
\end{equation*}
For the first identity, $\gamma$ does not actually need to be a geodesic; for the second identity, the fact that $\gamma$ is a geodesic eliminates an additional first order term compared to a general smooth curve. All geodesics $\gamma$ in the unit sphere are of the form
\begin{equation}\label{eq:geodesic_representation}
	\gamma(s)
	= \cos(cs)u + \sin(cs)v
\end{equation}
for some orthogonal vectors $u,v\in\S^{n-1}$, where $c\geq0$ is the constant speed of $\gamma$.

\paragraph{Distributions.}

For an open interval $(a,b)\subseteq \R$, we denote by $\D(a,b)$ the space of \emph{test functions} (that is, compactly supported smooth functions) on $(a,b)$, endowed with the standard Fréchet topology. The elements of the continuous dual space $\D'(a,b)$ are called \emph{distributions} on $(a,b)$. Moreover, we denote the pairing of a test function $\psi\in\D(a,b)$ and a distribution $g\in\D'(a,b)$ by $\left< \psi,g\right>_{\D'}$. The derivative of $g$ and the product of a smooth function $\eta\in C^\infty(a,b)$ with $g$ are defined by
\begin{equation*}
	\left< \psi, g' \right>_{\D'}
	= -\left< \psi' , g \right>_{\D'}
	\qquad\text{and}\qquad
	\left< \psi , \eta\cdot g\right>_{\D'}
	= \left< \eta\cdot \psi, g\right>_{\D'},
\end{equation*}

The space $C^{-\infty}(\S^{n-1})$ of distributions on the unit sphere is defined as the continuous dual space of the space $C^{\infty}(\S^{n-1})$ of smooth functions, endowed with the standard Fréchet topology. We denote the pairing of a test function $\phi\in C^\infty(\S^{n-1})$ and a distribution $\mu\in C^{-\infty}(\S^{n-1})$ by $\left< \phi,\mu\right>_{C^{-\infty}}$. By virtue of the spherical divergence theorem, we define the spherical gradient and spherical Hessian of a distribution $\mu\in C^{-\infty}(\S^{n-1})$ by 
\begin{equation*}
	\langle X, \nabla_{\S}\mu\rangle_{C^{-\infty}}
	= - \langle \Div_{\S}X, \mu\rangle_{C^{-\infty}}
	\qquad\text{and}\qquad
	\langle Y, \nabla_{\S}^2\mu\rangle_{C^{-\infty}}
	= \langle \Div_{\S}^2Y, \mu\rangle_{C^{-\infty}},
\end{equation*}
respectively, where $X$ is an arbitrary smooth vector field and $Y$ is an arbitrary smooth $2$-tensor field on $\S^{n-1}$. 

The group $\SO(n)$ acts on the space $C^\infty(\S^{n-1})$ in a natural way: for $\vartheta\!\in\SO(n)$ and $\phi\!\in C^\infty(\S^{n-1})$, we define $\vartheta\phi$ by $(\vartheta\phi)(u)=\phi(\vartheta^{-1}u)$. By duality, the action of $\SO(n)$ extends to distributions: for $\mu\in C^{-\infty}(\S^{n-1})$, we define $\vartheta\mu$ by $\left< \phi, \vartheta\mu\right>_{C^{-\infty}}=\left< \vartheta^{-1}\phi,\mu\right>_{C^{-\infty}}$. A map $\T$ is said to be \emph{$\SO(n)$ equivariant} if it intertwines rotations, that is, $\T(\vartheta\mu)=\vartheta \T\mu$ for every $\mu$ in the domain of $\T$.

We may identify the space $\M(a,b)$ of finite signed measures on $(a,b)$ with a subspace of $\D'(a,b)$. By virtue of to the \emph{Riesz-Markov-Kakutani representation theorem}, a distribution is defined by a finite signed measure on $(a,b)$ if and only if it is continuous on $\D(a,b)$ with respect to uniform convergence.
Similarly, the space $\M(\S^{n-1})$ of signed measures on $\S^{n-1}$ corresponds to the subspace of distributions which are continuous on $C^\infty(\S^{n-1})$ with respect to uniform convergence.

\paragraph{Harmonic analysis.}
 
Denote by $\mathcal{H}^n_k$ the space of \emph{spherical harmonics} of dimension $n$ and degree $k\geq 0$, that is, the space of harmonic, $k$-homogeneous polynomials on $\R^n$, restricted to the unit sphere $\S^{n-1}$. The spherical Laplacian $\Delta_{\S}=\tr \nabla_{\S}^2=\Div_{\S}\nabla_{\S}$ is a second-order uniformly elliptic self-adjoint operator on $\S^{n-1}$ that intertwines rotations.
It turns out that the spaces $\mathcal{H}^n_k$ are precisely the eigenspaces of $\Delta_{\S}$. Consequently, $L^2(\S^{n-1})$ decomposes into a direct orthogonal sum of them.
Each space $\mathcal{H}^n_k$ is a finite dimensional and irreducible $\SO(n)$ invariant subspace of $L^2(\S^{n-1})$ and for every $Y_k\in\mathcal{H}^n_k$, we have that $\Delta_{\S} Y_k= -k(k+n-2) Y_k$. For the box operator $\square_n=\frac{1}{n-1}\Delta_{\S}+\Id$, this implies
\begin{equation}\label{eq:box_eigenvalues}
	\square_n Y_k
	= - \frac{(k-1)(k+n-1)}{n-1}Y_k,
	\qquad Y_k\in\mathcal{H}^n_k.
\end{equation}

Throughout, we use $\bar{e}$ to denote a fixed but arbitrarily chosen pole of $\S^{n-1}$ and write $\SO(n-1)$ for the subgroup of rotations in $\SO(n)$ fixing $\e$. Functions, measures, and distributions on $\S^{n-1}$ that are invariant under the action of $\SO(n-1)$ are called \emph{zonal}. Clearly the value of a zonal function at $u\in\S^{n-1}$ depends only on the value of $\langle\e,u\rangle$, so there is a natural correspondence between zonal functions on $\S^{n-1}$ and functions on $[-1,1]$. For a zonal function $f\in C(\S^{n-1})$, we define $\bar{f}\in C[-1,1]$ by $f(u)=\bar{f}(\langle\e,u\rangle)$ and for $g\in C[-1,1]$, we define $\breve{g}\in C(\S^{n-1})$ by $\breve{g}(u)=g(\langle\e,u\rangle)$.

By identifying the unit sphere $\S^{n-1}$ with the homogeneous space $\SO(n)/\SO(n-1)$, the natural convolution structure on $C^\infty(\SO(n))$ can be used to define a convolution structure on $C^\infty(\S^{n-1})$. For an extensive exposition of this construction, we refer the reader to the excellent article by Grinberg and Zhang \cite{MR1658156}.
The spherical convolution $\phi\ast\nu$ of a smooth function $\phi\in C^\infty(\S^{n-1})$ and a zonal distribution $\nu\in C^{-\infty}(\S^{n-1})$ is defined by 
\begin{equation*}
	(\phi\ast \nu)(\vartheta\e)
	= \langle \phi, \vartheta \nu \rangle_{C^{-\infty}}
	= \langle \vartheta^{-1}\phi , \nu \rangle_{C^{-\infty}},
	\qquad\vartheta\in\SO(n).
\end{equation*}
Note that this definition does not depend on the special choice of $\vartheta$ and that $\phi\ast\nu\in C^\infty(\S^{n-1})$.

The convolution transform $\T_\nu:f\mapsto f\ast\nu$ is a self-adjoint endomorphism of $C^\infty(\S^{n-1})$ intertwining rotations and thus extends by duality to an endomorphism of $C^{-\infty}(\S^{n-1})$ which also intertwines rotations. That is, for a distribution $\mu\in C^{-\infty}(\S^{n-1})$,
\begin{equation*}
	\langle \phi, \mu\ast\nu \rangle_{C^{-\infty}}
	= \langle \phi, \T_\nu \mu \rangle_{C^{-\infty}}
	= \langle \T_\nu \phi, \mu \rangle_{C^{-\infty}}
	= \langle \phi\ast\nu, \mu \rangle_{C^{-\infty}}.
\end{equation*}
This definition includes the convolution of signed measures. 
Moreover, the convolution product is Abelian on zonal distributions. In the special case when $\phi\in C(\S^{n-1})$ and $f\in L^1(\S^{n-1})$ is zonal, the convolution product can be expressed as
\begin{equation*}
	(\phi\ast f)(u)
	= \int_{\S^{n-1}} \phi(v)\bar{f}(\langle u,v\rangle) dv,
	\qquad u\in\S^{n-1}.
\end{equation*}

For each $k\geq 0$, the space of zonal spherical harmonics in $\mathcal{H}^n_k$ is one-dimensional and spanned by $\breve{P}^n_k$, where $P^n_k\in C[-1,1]$ denotes the Legendre polynomial of dimension $n\geq 3$ and degree $k\geq 0$.
The orthogonal projection $\pi_k:L^2(\S^{n-1})\to\mathcal{H}^n_k$ onto the space $\mathcal{H}^n_k$ turns out to be the convolution transform associated with $\breve{P}^n_k$, that is,
\begin{equation*}
 	\pi_k \phi
 	= \tfrac{\dim\mathcal{H}^n_k}{\omega_{n-1}}~\phi \ast \breve{P}^n_k,
 	\qquad \phi\in L^2(\S^{n-1}),
\end{equation*}
where $\omega_n$ is the surface area of $\S^{n-1}$.
By duality, $\pi_k$ extends to a map from $C^{-\infty}(\S^{n-1})$ onto $\mathcal{H}^n_k$. Moreover, the formal Fourier series $\sum_{k=0}^\infty\pi_k \mu$ of a distribution $\mu\in C^{-\infty}(\S^{n-1})$ converges to $\mu$ in the weak sense. If $\nu\in C^{-\infty}(\S^{n-1})$ is zonal, then
\begin{equation*}
	\nu
	= \sum_{k=0}^\infty \tfrac{\dim\mathcal{H}^n_k}{\omega_{n-1}} a^n_k[\nu] \breve{P}^n_k,
\end{equation*}
where $a^n_k[\nu] = \langle \breve{P}^n_k, \nu \rangle_{C^{-\infty}}$. 

Throughout this work, we repeatedly use spherical cylinder coordinates $u=t\e+\sqrt{1-t^2}v$ on $\S^{n-1}$. For $\phi\in C(\S^{n-1})$ and $g\in C[-1,1]$,
\begin{equation}\label{eq:cylinder_coords}
	\int_{\S^{n-1}} \phi(u)g(\langle\e,u\rangle) du
	= \int_{(-1,1)} \int_{\S^{n-1}\cap\e} \phi(t\e+\sqrt{1-t^2}v)dv g(t) (1-t^2)^{\frac{n-3}{2}} dt.
\end{equation}
For a signed measure $\nu\in\M(\S^{n-1})$ that carries no mass at the poles, we denote by $\bar{\nu}\in\M(-1,1)$ the unique finite signed measure on $(-1,1)$ such that
\begin{equation*}
	\int_{\S^{n-1}} \breve{g}(u)\nu(du)
	= \omega_{n-1}\int_{(-1,1)} g(t)(1-t^2)^{\frac{n-3}{2}}\bar{\nu}(dt),
	\qquad g\in C[-1,1].
\end{equation*}
By \cref{eq:cylinder_coords}, this naturally extends the notation $\bar{f}$ for $f\in C(\S^{n-1})$. From the above, the Fourier coefficient $a^n_k[\nu]$ can be computed as
\begin{equation*}\label{eq:muliplier_formula}
	a^n_k[\nu]
	= \omega_{n-1}\int_{(-1,1)} P^n_k(t)(1-t^2)^{\frac{n-3}{2}}\bar{\nu}(dt).
\end{equation*}
The \emph{Funk-Hecke Theorem} states that the spherical harmonic expansion of the convolution product of a signed measure $\mu$ and a zonal signed measure $\nu$ is given by
\begin{equation*}
	\T_\nu\mu
	= \mu\ast\nu
	= \sum_{k=0}^\infty a^n_k[\nu]\pi_k \mu.
\end{equation*}
Hence the convolution transform $\T_\nu$ acts as a multiple of the identity on each space $\mathcal{H}^n_k$ of spherical harmonics. The Fourier coefficients $a^n_k[\nu]$ are called the \emph{multipliers} of $\T_\nu$.

For the explicit computations of multipliers, the following identity relating Legendre polynomials of different dimensions and degrees is useful:
\begin{equation}\label{eq:Legendre_deriv}
	\frac{d}{dt}P^n_k(t)
	= \frac{k(k+n-2)}{n-1} P^{n+2}_{k-1}(t).
\end{equation}
The Legendre polynomials also satisfy the following second-order differential equation, which also determines them up to a constant factor:
\begin{equation}\label{eq:Legendre_ODE}
	(1-t^2)\frac{d^2}{dt^2}P^n_k(t) - (n-1)t\frac{d}{dt}P^n_k(t) + k(k+n-2)P^n_k(t)
	= 0.
\end{equation}

\section{Regularity of the Spherical Convolution}
\label{sec:convol}

\subsection{Zonal Measures}
\label{sec:zonal_reg} 

In this section, we investigate the regularity of zonal signed measures. We provide necessary and sufficient conditions to decide whether $\nabla_{\S}\mu$ and $\nabla_{\S}^2\mu$ are signed measures and provide explicit formulas for them. As one might expect, this can be expressed in terms of the corresponding measure $\bar{\mu}$ on $(-1,1)$. In the smooth case, we have the following. 

\begin{lem}[{\cite{MR4316669}}]\label{zonal_smooth_deriv}
	Let $f\in C^\infty(\S^{n-1})$ be zonal. Then for all $u,v\in\S^{n-1}$,
	\begin{align}
		\nabla_{\S}f^v(u)
		&= \bar{f}'(\langle u,v\rangle)P_{u^\perp}v, 
		\label{eq:zonal_smooth_deriv1} \\
		\nabla_{\S}^2f^v(u)
		&= \bar{f}''(\langle u,v\rangle)(P_{u^\perp}v\otimes P_{u^\perp}v) - \langle u,v\rangle\bar{f}'(\langle u,v\rangle)P_{u^\perp}.
		\label{eq:zonal_smooth_deriv3} 
	\end{align}
\end{lem}

Throughout, $P_{u^\perp}$ denotes the orthogonal projection onto $u^\perp$, and $f^v$ denotes the rotated copy of $f$ with axis of revolution $v\in\S^{n-1}$, that is, $f^v=\vartheta f$ where $\vartheta\in\SO(n)$ is such that $\vartheta\e=v$.
Moreover, for every $v\in\S^{n-1}$, we define two operators $\J^v:C[-1,1]\to C(\S^{n-1})$ and $\J_v:C(\S^{n-1})\to C[-1,1]$ by
\begin{equation*}
	\J^v[\psi](u)
	= \psi(\langle u,v\rangle)
	\qquad\text{and}\qquad
	\J_v[\phi](t)
	= (1-t^2)^{\frac{n-3}{2}}\int_{\S^{n-1}\cap v^\perp} \phi(tv+\sqrt{1-t^2}w) dw.
\end{equation*}
By a change to spherical cylinder coordinates (see \cref{eq:cylinder_coords}), we obtain
\begin{equation*}
	\int_{\S^{n-1}} \phi(u)\J^v[\psi](u)du
	= \int_{[-1,1]} \J_v[\phi](t)\psi(t)dt,
\end{equation*}
which shows that $\J_v$ and $\J^v$ are adjoint to each other. Hence, by continuity and duality, both operators naturally extend to signed measures. With these notations in place, we prove the following dual version of \cref{zonal_smooth_deriv}.

\begin{lem}\label{tensor_field_deriv}
	Let $v\in\S^{n-1}$, let $X$ be a smooth vector field on $\S^{n-1}$, and $Y$ be a smooth $2$-tensor field on $\S^{n-1}$. Then for all $t\in (-1,1)$,
	\begin{align}
		\J_v[\Div_\S X](t)
		&= \frac{d}{dt}\J_v[\langle X, v\rangle](t), \label{eq:tensor_field_deriv1}\\
		\J_v[\Div_\S^2 Y](t)
		&= \frac{d^2}{dt^2}\J_v[\langle Y,v\otimes v\rangle ](t) + \frac{d}{dt}\left(t\J_v[\tr Y](t)\right).\label{eq:tensor_field_deriv2}
	\end{align}
\end{lem}
\begin{proof}
	Let $\psi\in\D(-1,1)$ be an arbitrary test function. Due to the spherical divergence theorem and \cref{eq:zonal_smooth_deriv1},
	\begin{equation*}
		\int_{\S^{n-1}} \Div_\S X(u) \psi(\langle u,v\rangle) du
		= - \int_{\S^{n-1}} \langle X(u), v \rangle \psi'(\langle u,v\rangle) du.
	\end{equation*}
	We transform both integrals to spherical cylinder coordinates.
	For the left hand side, we have
	\begin{equation*}
		\int_{\S^{n-1}} \Div_\S X(u) \psi(\langle u,v\rangle) du
		= \int_{(-1,1)} \J_v[\Div_\S X](t) \psi(t) dt
	\end{equation*}
	and for the right hand side,
	\begin{equation*}
		- \int_{\S^{n-1}} \langle X(u), v \rangle \psi'(\langle u,v\rangle) du
		= - \int_{(-1,1)} \J_v[\langle X, v\rangle]\psi'(t) dt
		= \int_{(-1,1)} \frac{d}{dt}\J_v[\langle X,v\rangle](t) \psi(t)dt,
	\end{equation*}
	where the final equality is obtained from integration by parts. This yields \cref{eq:tensor_field_deriv1}.
	
	For the second part of the lemma, let $\psi\in\D(-1,1)$ be an arbitrary test function. Due to the spherical divergence theorem for $2$-tensor fields and \cref{eq:zonal_smooth_deriv3},
	\begin{equation*}
		\int_{\S^{n-1}} \Div_\S^2 Y(u) \psi(\langle u,v\rangle)du
		= \int_{\S^{n-1}} \langle Y(u),v\otimes v\rangle \psi''(\langle u,v\rangle)-  \tr Y(u)\langle u,v\rangle \psi'(\langle u,v\rangle)du.
	\end{equation*}
	We transform both integrals to spherical cylinder coordinates. For the left hand side, we have
	\begin{equation*}
		\int_{\S^{n-1}} \Div_\S^2 Y(u) \psi(\langle u,v\rangle)du
		= \int_{(-1,1)} \J_v[\Div_\S^2 Y](t)\psi(t) dt,
	\end{equation*}
	and for the right hand side,
	\begin{align*}
		&\int_{\S^{n-1}} \langle Y(u),v\otimes v\rangle \psi''(\langle u,v\rangle)-  \tr Y(u)\langle u,v\rangle \psi'(\langle u,v\rangle)du \\
		&\qquad= \int_{(-1,1)} \J_v[\langle Y,v\otimes v\rangle](t)\psi''(t) - \J_v[\tr Y](t)t\psi'(t) dt \\
		&\qquad= \int_{(-1,1)} \left(\frac{d^2}{dt^2}\J_v[\langle Y,v\otimes v\rangle ](t) + \frac{d}{dt}\left(t\J_v[\tr Y](t)\right)\right) \psi(t) dt,
	\end{align*}
	where the final equality is obtained from integration by parts. This yields \cref{eq:tensor_field_deriv2}.
\end{proof}

Throughout Sections~\ref{sec:convol} and \ref{sec:w-mon}, we repeatedly apply the following two technical lemmas. Their proofs are given in \cref{sec:omitted}.

\begin{lem}\label{beta_int_by_parts}
	Let $\beta>0$ and $g\in\D'(-1,1)$ such that $(1-t^2)^{\frac{\beta}{2}}g'(t)\in\M(-1,1)$. Then $g$ is a locally integrable function and $(1-t^2)^{\frac{\beta-2}{2}}g(t)\in L^1(-1,1)$.	
	Moreover, whenever $\psi\in C^1(-1,1)$ is such that both $(1-t^2)^{-\frac{\beta-2}{2}}\psi'(t)$ and $(1-t^2)^{-\frac{\beta}{2}}\psi(t)$ are bounded on $(-1,1)$, then
	\begin{equation}\label{eq:beta_int_by_parts}
		\int_{(-1,1)} \psi(t)g'(dt)
		= - \int_{(-1,1)} \psi'(t)g(t)dt.
	\end{equation}
\end{lem}

\begin{lem}\label{polar_C^k_estimate}
	Let $v\in\S^{n-1}$, $w\in v^\perp$, and $\phi\in C^\infty(\S^{n-1}\backslash\{\pm v\})$. Then for all $k\geq 0$ and $\alpha, \beta\geq0$, there exists a constant $C_{n,k,\alpha,\beta}> 0$ such that for all $t\in (-1,1)$,
	\begin{equation}\label{eq:polar_C^k_estimate}
		\bigg| \frac{d^k}{dt^k}\J_v[\langle\cdot,w\rangle^\alpha(1-\langle\cdot,v\rangle^2)^{\frac{\beta}{2}}\phi](t) \bigg|
		\leq C_{n,k,\alpha,\beta} \abs{w}^\alpha\norm{\phi}_{C^k(\S^{n-1}\backslash\{\pm v\})} (1-t^2)^{\frac{n-3+\alpha+\beta}{2}-k}.
	\end{equation}
\end{lem}

For now, we only the need the following instances of \cref{polar_C^k_estimate}.

\begin{lem}\label{polar_C^k_estimate_XY}
	Let $v\in\S^{n-1}$, $X$ be a smooth vector field, and $Y$ be smooth $2$-tensor field on $\S^{n-1}$. Then for all $k\geq 0$, there exists a constant $C_{n,k}>0$ such that for all $t\in (-1,1)$,
	\begin{align}
		\bigg| \frac{d^k}{dt^k}\J_v[\langle X, v\rangle](t) \bigg|
		&\leq C_{n,k} \norm{X}_{C^k} (1-t^2)^{\frac{n-2}{2}-k},
		\label{eq:polar_C^k_estimate_X} \\
		\bigg| \frac{d^k}{dt^k}\J_v[\langle Y, v\otimes v\rangle](t) \bigg|
		&\leq C_{n,k} \norm{Y}_{C^k} (1-t^2)^{\frac{n-1}{2}-k}.
		\label{eq:polar_C^k_estimate_Y}
	\end{align}
\end{lem}
\begin{proof}
	For the proof of \cref{eq:polar_C^k_estimate_X}, note that $\langle X(u),v\rangle=(1-\langle u,v\rangle^2)^{\frac{1}{2}}\phi(u)$, where
	\begin{equation*}
		\phi(u)
		= \left< X(u),\frac{P_{u^\perp}v}{\abs{P_{u^\perp}v}}\right>,
		\qquad u\in\S^{n-1}\backslash\{\pm v\}.
	\end{equation*}
	Clearly, $\phi\in C^{\infty}(\S^{n-1}\backslash\{\pm v\})$, so we may apply \cref{eq:polar_C^k_estimate} for $\alpha=0$ and $\beta=1$. The proof of \cref{eq:polar_C^k_estimate_Y} is analogous.
\end{proof}

In the following proposition, we characterize the zonal signed measures for which their spherical gradient is a (vector-valued) signed measure and show that identity \cref{eq:zonal_smooth_deriv1} extends to this case in the weak sense.

\begin{prop}\label{zonal_reg_order1}
	Let $\mu\in \M(\S^{n-1})$ be zonal. Then $\nabla_{\S}\mu\in\M(\S^{n-1},\R^n)$ if and only if $\mu$ does not carry any mass at the poles and $(1-t^2)^{\frac{n-2}{2}}\bar{\mu}'(t)\in\M(-1,1)$. In this case, $\mu(du)=f(u)du$ for some zonal $f\in L^1(\S^{n-1})$ such that for all $v\in\S^{n-1}$,
	\begin{equation}\label{eq:zonal_reg_order1}
		\nabla_{\S}\mu^v(du)
		= P_{u^\perp}v (\J^v\bar{f}')(du).
	\end{equation}
\end{prop}
\begin{proof}
	First, let $X$ be a smooth vector field on $\S^{n-1}$ and note that if $\mu$ does not carry any mass on the poles or if $\supp X\subseteq\S^{n-1}\backslash\{\pm v\}$, then
	\begin{align}\label{eq:zonal_reg_order1:proof}
		\begin{split}
			&\left< X , \nabla_{\S}\mu^v \right>_{C^{-\infty}}
			= - \int_{\S^{n-1}} \Div_\S X(u) \mu^v(du)
			= - \int_{(-1,1)} \J_v[\Div_\S X](t) \bar{\mu}(dt) \\
			&\qquad= - \int_{(-1,1)} \frac{d}{dt}\J_v[\langle X, v \rangle](t) \bar{\mu}(dt),
		\end{split}		
	\end{align}
	where the second equality is obtained from a change to spherical cylinder coordinates and the final equality from \cref{eq:tensor_field_deriv1}.
	
	Suppose now that $\mu$ does not carry any mass on the poles and that $(1-t^2)^{\frac{n-2}{2}}\bar{\mu}'(t)\in\M(-1,1)$. Then $\bar{\mu}$ and thus $\mu$ is absolutely continuous, that is, $\mu(du)=f(u)du$ for some zonal $f\in L^1(\S^{n-1})$. By \cref{eq:polar_C^k_estimate_X}, we have that $(1-t^2)^{-\frac{n-2}{2}}\J_v[\langle X,v\rangle](t)$ and $(1-t^2)^{-\frac{n-4}{2}}\frac{d}{dt}\J_v[\langle X,v\rangle](t)$ are bounded, so for every smooth vector field $X$ on $\S^{n-1}$, \cref{eq:zonal_reg_order1:proof} and \cref{eq:beta_int_by_parts} yield
	\begin{equation*}
		\left< X , \nabla_{\S}\mu^v \right>_{C^{-\infty}}
		= \int_{(-1,1)} \J_v[\langle X,v\rangle](t) \bar{f}'(dt)
		= \int_{\S^{n-1}} \langle X(u) , P_{u^\perp}v\rangle (\J^v\bar{f}')(du),
	\end{equation*}
	where we applied a change to cylinder coordinates in the second equality. This proves identity \cref{eq:zonal_reg_order1} and in particular that $\nabla_{\S}\mu\in\M(\S^{n-1},\R^n)$.
	
	Conversely, suppose that $\nabla_{\S}\mu\in\M(\S^{n-1},\R^n)$. Take an arbitrary test function $\psi\in\D(-1,1)$ and define a smooth vector field $X$ by
	\begin{equation*}
		X(u)
		= \psi(u\cdot v)\frac{P_{u^\perp}v}{\abs{P_{u^\perp}v}},
		\qquad u\in\S^{n-1}.
	\end{equation*}
	Then $\supp X\subseteq\S^{n-1}\backslash\{\pm v\}$ and $\J_v[\langle X, v\rangle](t)=\omega_{n-1}(1-t^2)^{\frac{n-2}{2}}\psi(t)$, thus \cref{eq:zonal_reg_order1:proof} yields
	\begin{equation*}
		\omega_{n-1}\left< \psi(t) , (1-t^2)^{\frac{n-2}{2}}\bar{\mu}'(t) \right>_{\D'}
		= -\int_{(-1,1)} \frac{d}{dt}\J_v[\langle X, v\rangle](t) \bar{\mu}(dt)
		= \int_{\S^{n-1}} \langle X(u),\nabla_{\S}\mu^v(du)\rangle.
	\end{equation*}
	Therefore, we obtain the estimate
	\begin{equation*}
		\left| \left< \psi(t) , (1-t^2)^{\frac{n-2}{2}}\bar{\mu}'(t) \right>_{\D'} \right|
		\leq \omega_{n-1}^{-1} \norm{\nabla_{\S}\mu}_{\TV} \norm{X}_\infty 
		= \omega_{n-1}^{-1} \norm{\nabla_{\S}\mu}_{\TV} \norm{\psi}_{\infty},
	\end{equation*}
	where $\norm{\nabla_{\S}\mu}_{\TV}$ denotes the total variation of $\nabla_{\S}\mu$.
	Hence, $(1-t^2)^{\frac{n-2}{2}}\bar{\mu}'(t)\in\M(-1,1)$. Denoting $\mu_0 = \mathbbm{1}_{\S^{n-1}\backslash\{\pm\e\}}\mu$, the first part of the proof shows that $\nabla_{\S}\mu_0\in \M(\S^{n-1},\R^n)$, and thus,
	\begin{equation*}
		\mu(\{\e\})\nabla_{\S}\delta_{\e} + \mu(\{-\e\})\nabla_{\S}\delta_{-\e}
		= \nabla_{\S}\mu - \nabla_{\S}\mu_0
		\in \M(\S^{n-1},\R^n).
	\end{equation*}
	Since $\nabla_{\S}\delta_{\e}$ and $\nabla_{\S}\delta_{-\e}$ are distributions of order one (see, e.g., \cite[Section~2.1]{MR1065136}),	this is clearly possible only if $\mu$ carries no mass at the poles.
\end{proof}	

Employing the same technique as in \cref{zonal_reg_order1}, we can characterize signed measures for which their spherical Hessian is a (matrix-valued) signed measure. Identities \cref{eq:zonal_smooth_deriv1} and \cref{eq:zonal_smooth_deriv3} extend to this case in the weak sense.

\begin{prop}\label{zonal_reg_order2}
	Let $\mu\in\M(\S^{n-1})$ be zonal. Then $\nabla_{\S}^2 \mu\in\M(\S^{n-1},\R^n)$ if and only if $\mu$ carries no mass at the poles and $(1-t^2)^{\frac{n-1}{2}}\bar{\mu}''(t)\in\M(-1,1)$. In this case, $\mu(du)=f(u)du$ for some zonal $f\in L^1(\S^{n-1})$ such that $\nabla_{\S}\mu\in L^1(\S^{n-1},\R^n)$ and for all $v\in\S^{n-1}$,
	\begin{align}
		\nabla_{\S}\mu^v(du)
		&= P_{u^\perp}v \bar{f}'(\langle u,v\rangle),\label{eq:zonal_reg_order2_1}\\
		\nabla_{\S}^2 \mu^v(du)
		&= (P_{u^\perp}v \otimes P_{u^\perp}v)(\J^v\bar{f}'')(du) - \langle u,v\rangle\bar{f}'(\langle u,v\rangle)P_{u^\perp}du.\label{eq:zonal_reg_order2_2}
	\end{align}
\end{prop}
\begin{proof}
	First, take a smooth $2$-tensor field $Y$ on $\S^{n-1}$ and note that if $\mu$ does not carry any mass on the poles or if $\supp Y\subseteq\S^{n-1}\backslash\{\pm v\}$, then
	\begin{align}\label{eq:zonal_reg_order2:proof}
		\begin{split}
			&\langle Y , \nabla_{\S}^2 \mu^v \rangle_{C^{-\infty}}
			= \int_{\S^{n-1}} \Div_{\S}^2Y(u)\mu^v(du)
			= \int_{(-1,1)} \J_v[\Div_{\S}^2Y](t)\bar{\mu}(dt) \\
			&\qquad= - \left< \frac{d}{dt}\J_v[\langle Y,v\otimes v\rangle](t) + t\J_v[\tr Y](t), \bar{\mu}'(t) \right>_{\D'}, 
		\end{split}
	\end{align}
	where the second equality is obtained from a change to spherical cylinder coordinates and the final equality from \cref{eq:tensor_field_deriv2}.
	
	Suppose now that $\mu$ carries no mass at the poles and that $(1-t^2)^{\frac{n-1}{2}}\bar{\mu}''(t)\in\M(-1,1)$. Then $\bar{\mu}$, and thus, $\mu$ are absolutely continuous, that is, $\mu(du)=f(u)du$ for some zonal $f\in L^1(\S^{n-1})$. Moreover, \cref{beta_int_by_parts} implies that $(1-t^2)^{\frac{n-3}{2}}\bar{f}'(t)\in L^1(-1,1)$, and thus, \cref{zonal_reg_order1} implies $\nabla_{\S}\mu\in L^1(\S^{n-1},\R^n)$ and identity \cref{eq:zonal_reg_order2_1}. By \cref{eq:polar_C^k_estimate_Y}, we have that $(1-t^2)^{-\frac{n-1}{2}}\J_v[\langle Y,v\otimes v\rangle](t)$ and $(1-t^2)^{-\frac{n-3}{2}}\frac{d}{dt}\J_v[\langle Y,v\otimes v\rangle](t)$ are bounded, so for every smooth $2$-tensor field $Y$ on $\S^{n-1}$, \cref{eq:zonal_reg_order2:proof} and \cref{eq:beta_int_by_parts} yield
	\begin{align*}
	 	&\langle Y , \nabla_{\S}^2 \mu^v \rangle_{C^{-\infty}}
	 	= \int_{(-1,1)} \J_v[\langle Y,v\otimes v\rangle](t)\bar{f}''(dt) - \int_{(-1,1)} t\J_v[\tr Y](t) \bar{f}'(t) dt \\
	 	&\qquad= \int_{\S^{n-1}} \langle Y(u) , P_{u^\perp}v \otimes P_{u^\perp}v\rangle (\J^v\bar{f}'')(du) - \int_{\S^{n-1}} \langle Y(u) , P_{u^\perp}\rangle \langle u, v\rangle \bar{f}'(\langle u,v\rangle)du,
	\end{align*}
 	where we applied a change to cylinder coordinates in the second equality. This proves \cref{eq:zonal_reg_order2_2} and in particular that $\nabla_{\S}^2 \mu\in\M(\S^{n-1},\R^n)$.
 	
 	Conversely, suppose now that $\nabla_{\S}^2 \mu\in\M(\S^{n-1},\R^n)$. Take an arbitrary test function $\psi\in\D(-1,1)$ and define a smooth $2$-tensor field $Y$ on $\S^{n-1}$ by
 	\begin{equation*}
 		Y(u)
 		= \psi(\langle u,v\rangle) \left( \frac{P_{u^\perp}v}{\abs{P_{u^\perp}v}}\otimes \frac{P_{u^\perp}v}{\abs{P_{u^\perp}v}} - \frac{1}{n-2} \left(P_{u^\perp} - \frac{P_{u^\perp}v}{\abs{P_{u^\perp}v}}\otimes \frac{P_{u^\perp}v}{\abs{P_{u^\perp}v}}\right) \right),
 		\qquad u\in\S^{n-1}.
 	\end{equation*}
 	Then $\supp Y\subseteq\S^{n-1}\backslash\{\pm v\}$, and $Y$ satisfies $\J_v[\langle Y,v\otimes v\rangle](t)=\omega_{n-1}(1-t^2)^{\frac{n-1}{2}}\psi(t)$ and $\tr Y=0$, thus 	
 	\begin{align*}
 		&\omega_{n-1}\left< \psi(t) , (1-t^2)^{\frac{n-1}{2}}\bar{\mu}''(t) \right>_{\D'}
 		= \omega_{n-1}\left< \J_v[Y,v\otimes v],\bar{\mu}''(t)\right>_{\D'} \\
 		&\qquad= -\left< \frac{d}{dt}\J_v[\langle Y,v\otimes v\rangle](t) + t\J_v[\tr Y](t), \bar{\mu}'(t) \right>_{\D'}
 		= \int_{\S^{n-1}} \langle Y(u) , \nabla_{\S}^2\mu^v(du) \rangle.
 	\end{align*}
 	Therefore, we obtain the estimate:
 	\begin{equation*}
 		\left| \left< \psi(t) , (1-t^2)^{\frac{n-1}{2}}\bar{\mu}''(t) \right>_{\D'} \right|
 		\leq \omega_{n-1}^{-1} \norm{\nabla_{\S}^2\mu}_{\TV} \norm{Y}_\infty 
 		= \omega_{n-1}^{-1} \norm{\nabla_{\S}^2\mu}_{\TV} \norm{\psi}_{\infty},
 	\end{equation*}
	where $\norm{\nabla_{\S}^2\mu}_{\TV}$ denotes the total variation of $\nabla_{\S}^2\mu$.
 	Hence, $(1-t^2)^{\frac{n-1}{2}}\bar{\mu}''(t)\in\M(-1,1)$. Denoting $\mu_0=\mathbbm{1}_{\S^{n-1}\backslash\{\pm\e\}}\mu$, the first part of the proof shows that $\nabla_{\S}^2\mu_0 \in\M(\S^{n-1},\R^{n\times n})$, and thus,
 	\begin{equation*}
 		\mu(\{\e\})\nabla_{\S}^2\delta_{\e} + \mu(\{-\e\})\nabla_{\S}^2\delta_{-\e}
 		= \nabla_{\S}^2\mu - \nabla_{\S}^2\mu_0
 		\in \M(\S^{n-1},\R^{n\times n}).
 	\end{equation*}
 	Since $\nabla_{\S}^2\delta_{\e}$ and $\nabla_{\S}^2\delta_{-\e}$ are distributions of order two (see, e.g., \cite[Section~2.1]{MR1065136}), this is clearly possible only if $\mu$ carries no mass at the poles.
\end{proof}

For later purposes, it will be useful to describe the regularity of zonal functions $f\in L^1(\S^{n-1})$ in terms of their Laplacian. 

\begin{lem}\label{zonal_Laplacian}
	If $f\in L^1(\S^{n-1})$ is zonal and $\Delta_{\S}f\in\M(\S^{n-1})$, then for almost all $t\in(-1,1)$,
	\begin{equation}\label{eq:zonal_Laplacian}
		(\Delta_{\S}f)(\{u\in\S^{n-1}:\langle \e,u\rangle>t \})
		= - \omega_{n-1}(1-t^2)^{\frac{n-1}{2}}\bar{f}'(t).
	\end{equation}
\end{lem}
\begin{proof}
	Let $\psi\in\D(-1,1)$ be an arbitrary test function. Define $\eta(t)=\int_{(-1,t)} \psi(s) ds$ and note that $\eta(\langle\e,\cdot\rangle)\in C^\infty(\S^{n-1})$. Then Lebesgue-Stieltjes integration by parts yields
	\begin{align*}
		&\int_{(-1,1)} (\Delta_{\S}f)(\{u\in\S^{n-1}:\langle \e,u\rangle>t \})\psi(t) dt
		= \int_{[-1,1]} \J_{\e} [\Delta_{\S}f]((t,1]) \psi(t) dt \\
		&\qquad= \int_{[-1,1]} \eta(t) \J_{\e} [\Delta_{\S}f](dt)
		= \int_{\S^{n-1}} \eta(\langle\e,u\rangle) (\Delta_{\S}f)(du)
		= \int_{\S^{n-1}} \Delta_{\S}\eta(\langle\e,\cdot\rangle)(u) f(u)du,
	\end{align*}
	where the third equality follows from the characteristic property of the pushforward measure and the final equality, from the definition of the distributional spherical Laplacian. Taking the trace\linebreak in \cref{eq:zonal_smooth_deriv3},
	\begin{equation*}
		\Delta_{\S}\eta(\langle\e,\cdot\rangle)(u)
		= (1-\langle\e,u\rangle^2)\eta''(\langle\e,u\rangle) - (n-1)\langle\e,u\rangle\eta'(\langle\e,u\rangle).
	\end{equation*}
	By a change to spherical cylinder coordinates, we obtain
	\begin{align*}
		&\int_{(-1,1)} (\Delta_{\S}f)(\{u\in\S^{n-1}:\langle \e,u\rangle>t \})\psi(t) dt \\
		&\qquad = \omega_{n-1}\int_{(-1,1)} \left((1-t^2)\eta''(t) - (n-1)t\eta'(t)\right) (1-t^2)^{\frac{n-3}{2}}\bar{f}(t)dt \\
		&\qquad = \omega_{n-1}\int_{(-1,1)} \frac{d}{dt}\left((1-t^2)^{\frac{n-1}{2}}\psi(t)\right) \bar{f}(t)dt
		= -\omega_{n-1}\left< \psi(t), (1-t^2)^{\frac{n-1}{2}}\bar{f}'(t) \right>_{\D'}.
	\end{align*}
	Since $\psi$ was arbitrary, identity \cref{eq:zonal_Laplacian} holds for almost all $t\in (-1,1)$.
\end{proof}

From now on, we denote by
\begin{equation}\label{eq:SCap}
	\SCap_r(u)
	= \{v\in\S^{n-1}: \langle u,v\rangle > \cos r \}
\end{equation}
the \emph{spherical cap} around $u\in\S^{n-1}$ with radius $r\geq0$. The following proposition classifies zonal functions $f$ for which their spherical Hessian is a signed measure in terms of the behavior of $\Delta_{\S}f$ on small polar caps.

\begin{prop}\label{zonal_reg_equi}
	For a zonal function $f\in L^1(\S^{n-1})$, the following are equivalent:
	\begin{enumerate}[label=\upshape(\alph*),topsep=1.0ex,itemsep=0.5ex]
		\item \label{zonal_reg_equi:Hessian}
		$\nabla_{\S}^2 f\in\M(\S^{n-1},\R^{n\times n})$,
		\item \label{zonal_reg_equi:Laplacian}
		$\Delta_{\S} f \in \M(\S^{n-1})$ and $\int_{(0,\frac{\pi}{2})} \frac{\abs{(\Delta_{\S}f)(\SCap_r(\pm\e))}}{r}dr
		<\infty$,
		\item \label{zonal_reg_equi:box}
		$\square_n f \in \M(\S^{n-1})$ and  $\int_{(0,\frac{\pi}{2})} \frac{\abs{(\square_n f)(\SCap_r(\pm\e))}}{r}dr
		<\infty$.
	\end{enumerate}
\end{prop}
\begin{proof}
	Each of the three statements above implies that $\Delta_{\S}f\in\M(\S^{n-1})$. Due to \cref{zonal_Laplacian}, we have that $\nabla_{\S}f\in L^1(\S^{n-1},\R^n)$ and for almost all $t\in (-1,1)$,
	\begin{equation}\label{eq:zonal_reg_equi:proof1}
		\J_{\e}[\Delta_{\S}f]((t,1])
		= - \omega_{n-1}(1-t^2)^{\frac{n-1}{2}}\bar{f}'(t).
	\end{equation}
	Taking the distributional derivative on both sides yields
	\begin{equation}\label{eq:zonal_reg_equi:proof2}
		\mathbbm{1}_{(-1,1)}(t)\J_{\e}[\Delta_{\S}f](dt)
		= \omega_{n-1}\left( (1-t^2)^{\frac{n-1}{2}}\bar{f}''(t) - (n-1)(1-t^2)^{\frac{n-3}{2}}\bar{f}'(t)\right)
	\end{equation}
	in $\D'(-1,1)$. According to \cref{zonal_reg_order2}, condition \ref{zonal_reg_equi:Hessian} is fulfilled if and only if $(1-t^2)^{\frac{n-1}{2}}\bar{f}''(t)$ is a finite signed measure. Due to \cref{eq:zonal_reg_equi:proof1} and \cref{eq:zonal_reg_equi:proof2}, this is the case if and only if
	\begin{equation*}
		\int_{(0,1)} \frac{\abs{(\Delta_{\S}f)(\{u\in\S^{n-1}:\langle\e,u\rangle>t\})}}{1-t^2} dt +\int_{(0,1)} \frac{\abs{(\Delta_{\S}f)(\{u\in\S^{n-1}:\langle-\e,u\rangle>t\})}}{1-t^2} dt
		< \infty.
	\end{equation*}
	The substitution $t=\cos r$ then shows that	conditions \ref{zonal_reg_equi:Hessian} and \ref{zonal_reg_equi:Laplacian} are equivalent.
	
	For the equivalence of \ref{zonal_reg_equi:Laplacian} and \ref{zonal_reg_equi:box}, it suffices to show that
	 $(1-t^2)^{-1}\big|\int_{\{u:\abs{\langle\e,u\rangle}>t\}} f(u) du\big|$ is integrable on $(0,1)$. To that end, using spherical cylinder coordinates, we estimate
	\begin{equation*}
		\int_{\{u:\abs{\langle\e,u\rangle}>t\}} \abs{f(u)}du
		= \int_{(t,1)} (1-s^2)^{\frac{n-3}{2}}\abs{\bar{f}(s)}ds
		\leq (1-t^2)^{\frac{1}{2}} \int_{(-1,1)} (1-s^2)^{\frac{n-4}{2}}\abs{\bar{f}(s)}ds,
	\end{equation*}
	for $t\in (0,1)$. Since $\nabla_{\S}f\in L^1(\S^{n-1},\R^n)$, \cref{zonal_reg_order1} and \cref{beta_int_by_parts} imply that $(1-t^2)^{\frac{n-4}{2}}\bar{f}(t)$ is integrable. This completes the proof.
\end{proof}

We want to note that \cref{zonal_Laplacian} and \cref{zonal_reg_equi} still hold if the zonal function $f$ is replaced by a zonal signed measure that carries no mass at the poles.

\begin{exl}\label{exl:Berg_fct_regularity}
	For Berg's function $g_n$ we have that $\breve{g}_n=g_n(\langle\e,\cdot\rangle)$ is an integrable function on the sphere and $\square_n\breve{g}_n=(\Id-\pi_1)\delta_{\e}$ is a finite signed measure. Hence \cref{zonal_Laplacian} implies that $\nabla_{\S}\breve{g}_n$ is integrable on $\S^{n-1}$. At the same time, the integrability condition in \cref{zonal_reg_equi}~\ref{zonal_reg_equi:box} is clearly violated, so the distributional spherical Hessian $\nabla_{\S}^2\breve{g}_n$ is not a finite signed measure.
\end{exl}

\subsection{Convolution Transforms}
\label{sec:bounded_ops}

Linear operators on functions on the unit sphere intertwining rotations can be identified with convolution transforms, as the following theorem shows.

\begin{thm}[{\cite{MR2343256}}]\label{convol_trafo_C^0}
	If $\mu\in\M(\S^{n-1})$ is zonal, then the convolution transform $\T_\mu$ is a bounded linear operator on $C(\S^{n-1})$.	
	Conversely, if $\T$ is an $\SO(n)$ equivariant bounded linear operator on $C(\S^{n-1})$, then there exists a unique zonal $\mu\in\M(\S^{n-1})$ such that $\T=\T_\mu$.
\end{thm}

The regularizing properties of a convolution transform correspond to the regularity of its integral kernel. In this section, we classify rotation equivariant bounded linear operators from $C(\S^{n-1})$ to $C^2(\S^{n-1})$ in terms of their integral kernel, proving \cref{convol_trafo_C^2:intro}. We require the following two lemmas.

\begin{lem}\label{curve_deriv}
	Let $\phi\in C^\infty(\S^{n-1})$ and let $\gamma: I\subseteq\R\to \S^{n-1}$ be a smooth curve in $\S^{n-1}$. Then for all $s\in I$ and $t\in(-1,1)$,
	\begin{align}
		\frac{d}{ds} \J_{\gamma(s)}[\phi](t)
		&=  - \frac{d}{dt} \J_{\gamma(s)}[\langle\cdot, \gamma'(s)\rangle\phi](t),\label{eq:curve_deriv1}\\
		\frac{d}{ds} \J_{\gamma(s)}[\langle\cdot,\gamma'(s)\rangle\phi](t)
		&=  - \frac{d}{dt}\J_{\gamma(s)}[\langle\cdot,\gamma'(s)\rangle^2\phi](t) + \J_{\gamma(s)}[\langle\cdot,\gamma''(s)\rangle\phi](t).\label{eq:curve_deriv2}
	\end{align}
\end{lem}
\begin{proof}
	Let $\psi\in\D(-1,1)$ be an arbitrary test function. On the one hand, spherical cylinder coordinates yield
	\begin{align*}
		\frac{d}{ds} (\phi\ast\breve{\psi})(\gamma(s))
		= \frac{d}{ds}\int_{(-1,1)} \J_{\gamma(s)}[\phi](t) \psi(t) dt
		= \int_{(-1,1)} \frac{d}{ds}\J_{\gamma(s)}[\phi](t) \psi(t) dt.
	\end{align*}
	On the other hand,
	\begin{align*}
		\frac{d}{ds} (\phi\ast\breve{\psi})(\gamma(s))
		= \frac{d}{ds} \int_{\S^{n-1}} \phi(v) \J^v [\psi](\gamma(s)) dv
		= \int_{\S^{n-1}} \phi(v) \langle \nabla_{\S}\J^v[\psi](\gamma(s)),\gamma'(s)\rangle dv.
	\end{align*}
	By \cref{eq:zonal_smooth_deriv1} and a change to spherical cylinder coordinates, we obtain
	\begin{align*}
		&\frac{d}{ds} (\phi\ast\breve{\psi})(\gamma(s))
		= \int_{\S^{n-1}} \phi(v) \langle v, \gamma'(s) \rangle \psi'(\langle \gamma(s),v\rangle)dv
		= \int_{(-1,1)} \J_{\gamma(s)}[\phi\langle \cdot, \gamma'(s) \rangle](t) \psi'(t) dt \\
		&\qquad= - \int_{(-1,1)} \frac{d}{dt}\J_{\gamma(s)}[\phi\langle \cdot,\gamma'(s) \rangle](t) \psi(t) dt,
	\end{align*}
	where the final equality follows from integration by parts. This implies \cref{eq:curve_deriv1}. 
	
	For the second part of the lemma, let $\psi\in\D(-1,1)$ be an arbitrary test function. On the one hand, spherical cylinder coordinates yield
	\begin{equation*}
		\frac{d}{ds} ((\langle\cdot,\gamma'(s)\rangle \psi)\ast\breve{\phi})(\gamma(s))
		= \frac{d}{ds}\int_{(-1,1)} \J_{\gamma(s)}[\langle\cdot,\gamma'(s)\rangle\phi](t)\psi(t) dt
		= \int_{(-1,1)}\frac{d}{ds} \J_{\gamma(s)}[\langle\cdot,\gamma'(s)\rangle\phi](t)\psi(t) dt.
	\end{equation*}
	On the other hand,
	\begin{align*}
		&\frac{d}{ds} ((\langle\cdot,\gamma'(s)\rangle \psi)\ast\breve{\phi})(\gamma(s))
		=  \frac{d}{ds}\int_{\S^{n-1}} \phi(v)\langle v, \gamma'(s) \rangle \J^v[\psi](\gamma(s)) dv \\
		&\qquad = \int_{\S^{n-1}} \phi(v) \langle v, \gamma'(s)\rangle \langle \nabla_{\S}\J^v[\psi](\gamma(s)),\gamma'(s)\rangle dv + \int_{\S^{n-1}} \phi(v)\langle v, \gamma''(s) \rangle \psi(\langle v, \gamma(s)\rangle) dv.
	\end{align*}	
	By \cref{eq:zonal_smooth_deriv1} and a change to cylinder coordinates, we obtain
	\begin{align*}
		&\frac{d}{ds} ((\langle\cdot,\gamma'(s)\rangle \psi)\ast\breve{\phi})(\gamma(s)) \\
		&\qquad=  \int_{\S^{n-1}} \phi(v) \langle v, \gamma'(s)\rangle^2 \psi'(\langle v,\gamma(s)\rangle) dv + \int_{\S^{n-1}} \phi(v)\langle v, \gamma''(s) \rangle \psi(\langle v, \gamma(s)\rangle) dv\\
		&\qquad= \int_{(-1,1)} \J_{\gamma(s)}[\langle\cdot,\gamma'(s)\rangle^2\phi](t)\psi'(t)dt + \int_{(-1,1)} \J_{\gamma(s)}[\langle\cdot,\gamma''(s)\rangle \phi](t) \psi(t)dt \\
		&\qquad= \int_{(-1,1)} \left(-\frac{d}{dt}\J_{\gamma(s)}[\langle\cdot,\gamma'(s)\rangle^2\phi](t) +  \J_{\gamma(s)}[\langle\cdot,\gamma''(s)\rangle \phi](t)\right) \psi(t)dt,
	\end{align*}
	where the final equality follows from integration by parts. This implies \cref{eq:curve_deriv2}.
\end{proof}

\begin{lem}\label{curve_estimate}
	Let $\phi\in C^\infty(\S^{n-1})$ and let $\gamma:I\subseteq\R\to\S^{n-1}$ be a smooth curve in $\S^{n-1}$. Then for all $s\in I$ and $t\in (-1,1)$,
	\begin{align}
		\bigg| \frac{d}{ds}\J_{\gamma(s)}[\phi](t) \bigg|
		&\leq C_n\abs{\gamma''(s)}\norm{\phi}_{C^1(\S^{n-1})} (1-t^2)^{\frac{n-4}{2}},
		\label{eq:curve_estimate1} \\
		\bigg| \frac{d}{ds}\J_{\gamma(s)}[\langle\cdot,\gamma'(s)\rangle\phi](t) \bigg|
		&\leq C_n\left(\abs{\gamma'(s)}^2+\abs{\gamma''(s)}\right)\norm{\phi}_{C^1(\S^{n-1})}(1-t^2)^{\frac{n-3}{2}}.
		\label{eq:curve_estimate2}
	\end{align}
\end{lem}
\begin{proof}
	For the proof of \cref{eq:curve_estimate1}, apply estimate \cref{eq:polar_C^k_estimate} to  the right hand side of \cref{eq:curve_deriv1} in the instance where $(k,\alpha,\beta)=(1,1,0)$.
	To obtain \cref{eq:curve_estimate2}, apply estimate \cref{eq:polar_C^k_estimate} to the right hand side of \cref{eq:curve_deriv2} in the instances where $(k,\alpha,\beta)=(1,2,0)$ and $(k,\alpha,\beta)=(0,0,0)$.
\end{proof}

\begin{thm}\label{convol_trafo_C^1}
	If $f\in L^1(\S^{n-1})$ is zonal and $\nabla_{\S}f \in\M(\S^{n-1},\R^n)$, then the convolution transform $\T_f$ is a bounded linear operator from $C(\S^{n-1})$ to $C^1(\S^{n-1})$.
	Conversely, if $\T$ is an $\SO(n)$ equivariant bounded linear operator from $C(\S^{n-1})$ to $C^1(\S^{n-1})$, then there exists a unique zonal $f\in L^1(\S^{n-1})$ satisfying $\nabla_{\S}f \in\M(\S^{n-1},\R^n)$ such that $\T=\T_f$.
	In this case, for every $\phi\in C(\S^{n-1})$ and $u\in\S^{n-1}$,
	\begin{equation}\label{eq:convol_trafo_C^1}
		\nabla_{\S}(\phi\ast f)(u)
		= \int_{\S^{n-1}} \phi(v)P_{u^\perp}v (\J^u\bar{f}')(dv).
	\end{equation}
\end{thm}
\begin{proof}
	Suppose that $f\in L^1(\S^{n-1})$ is zonal and that $\nabla_{\S}f \in\M(\S^{n-1},\R^n)$. \cref{convol_trafo_C^0} implies that $\T_f$ is a bounded linear operator on $C(\S^{n-1})$. First, we will verify identity \cref{eq:convol_trafo_C^1} for an arbitrary smooth function $\phi\in C^\infty(\S^{n-1})$. Take a point $u\in\S^{n-1}$ and a tangent vector $w\in u^\perp$. Choosing a smooth curve $\gamma$ in $\S^{n-1}$ such that $\gamma(0)=u$ and $\gamma'(0)=w$ yields
	\begin{equation*}
		\langle \nabla_{\S}(\phi\ast f)(u) , w \rangle
		= \left.\frac{d}{ds}\right|_0 (\phi\ast f)(\gamma(s))
		= \left.\frac{d}{ds}\right|_0 \int_{(-1,1)} \J_{\gamma(s)}[\phi](t) \bar{f}(t) dt.
	\end{equation*}
	Note that \cref{zonal_reg_order1} and \cref{beta_int_by_parts} imply that $(1-t^2)^{\frac{n-2}{2}}\bar{f}'(t)$ is a finite signed measure and $(1-t^2)^{\frac{n-4}{2}}\bar{f}(t)$ is an integrable function on $(-1,1)$. Due to the estimate \cref{eq:curve_estimate1}, we have that $(1-t^2)^{-\frac{n-4}{2}}\frac{d}{ds} \J_{\gamma(s)}[\phi](t)$ is bounded uniformly in $s$ for all sufficiently small $s$, so we may interchange differentiation and integration and obtain
	\begin{align*}
		&\langle \nabla_{\S}(\phi\ast f)(u) , w \rangle
		= \int_{(-1,1)}\left.\frac{d}{ds} \right|_0 \J_{\gamma(s)}[\phi](t) \bar{f}(t) dt
		= - \int_{(-1,1)} \frac{d}{dt} \J_u[\langle\cdot,w\rangle\phi](t) \bar{f}(t) dt \\
		&\qquad= \int_{(-1,1)} \J_u[\langle\cdot,w\rangle\phi](t) \bar{f}'(dt)
		= \int_{\S^{n-1}} \phi(v)\langle P_{u^\perp}v,w\rangle(\J^u\bar{f}')(dv),
	\end{align*}
	where the second equality follows from \cref{eq:curve_deriv1}, the third from \cref{eq:beta_int_by_parts}, and the final equality from a change to spherical cylinder coordinates. This proves identity \cref{eq:convol_trafo_C^1} for all $\phi\in C^\infty(\S^{n-1})$.
	
	As an immediate consequence, $\norm{\T_f \phi}_{C^1(\S^{n-1})} \leq C\norm{\phi}_{C(\S^{n-1})}$ for some constant $C\geq 0$ and all $\phi\in C^\infty(\S^{n-1})$. Thus, $\T_f$ extends to a bounded linear operator $\overline{\T}: C(\S^{n-1})\to C^1(\S^{n-1})$. Since $\T_f:C(\S^{n-1})\to C(\S^{n-1})$ is a bounded operator and the inclusion $C^1(\S^{n-1})\subseteq C(\S^{n-1})$ is continuous, $\overline{\T}$ agrees with $\T_f$. By density and continuity, \cref{eq:convol_trafo_C^1} is valid for all $\phi\in C(\S^{n-1})$.
	
	For the second part of the theorem, suppose that $\T$ is an $\SO(n)$ equivariant bounded linear operator from $C(\S^{n-1})$ to $C^1(\S^{n-1})$. Then \cref{convol_trafo_C^0} implies that $\T=\T_\mu$ for a unique zonal $\mu\in\M(\S^{n-1})$. According to \cref{zonal_reg_order1}, it suffices to show that $\mu$ carries no mass at the poles and that $(1-t^2)^{\frac{n-2}{2}}\bar{\mu}'(t)\in\M(-1,1)$. To that end, take an arbitrary test function $\psi\in\D(-1,1)$, a point $u\in\S^{n-1}$, a unit tangent vector $w\in u^\perp$, and define
	\begin{equation*}
		\phi (v) = \frac{\langle v,w\rangle}{\sqrt{1-\langle v,u\rangle^2}}\psi(\langle v,u\rangle),
		\qquad v\in\S^{n-1}.
	\end{equation*}
	Then $\phi$ is a smooth function satisfying $\J_u[\langle\cdot,w\rangle\phi](t)=C_{n}(1-t^2)^{\frac{n-2}{2}}\psi(t)$, where $C_n>0$ is given by $C_{n}
	= \int_{\S^{n-1}\cap u^\perp} \langle w,v\rangle^2 dv$.
	Choosing a smooth curve $\gamma$ in $\S^{n-1}$ such that $\gamma(0)=u$ and $\gamma'(0)=w$ yields
	\begin{align*}
		&C_{n} \left< \psi(t) , (1-t^2)^{\frac{n-2}{2}}\bar{\mu}'(t) \right>_{\D'}
		= \left< \J_u[\langle\cdot,w\rangle\phi](t), \bar{\mu}'(t) \right>_{\D'} \\
		&\qquad= - \int_{(-1,1)} \frac{d}{dt} \J_u[\langle\cdot,w\rangle\phi](t) \bar{\mu}(dt)
		= \int_{(-1,1)} \left.\frac{d}{ds}\right|_0 \J_{\gamma(s)}[\phi](t) \bar{\mu}(dt),
	\end{align*}
	where the final equality is due to \cref{eq:curve_deriv1}.
	Observe	that $\supp \J_{\gamma(s)}[\phi]\subseteq [-1+\varepsilon,1-\varepsilon]$ and that $\abs{\frac{d}{ds}\J_{\gamma(s)}[\phi]}\leq C$ uniformly in $s$ for all sufficiently small $s$, so we may interchange differentiation and integration and obtain
	\begin{equation*}
		C_{n} \left< \psi(t) , (1-t^2)^{\frac{n-2}{2}}\bar{\mu}'(t) \right>_{\D'}
		= \left.\frac{d}{ds}\right|_0 \int_{(-1,1)} \J_{\gamma(s)}[\phi_j](t) \bar{\mu}(dt)
		= \langle \nabla_{\S}(\phi\ast \mu)(u), w\rangle.
	\end{equation*}
	Therefore, we arrive at the following estimate:
	\begin{align*}
		\left| \left< \psi(t) , (1-t^2)^{\frac{n-2}{2}}\bar{\mu}'(t) \right>_{\D'}\right|
		\leq C_{n}^{-1}\norm{\T}\norm{\phi}_{C(\S^{n-1})}
		= C_{n}^{-1}\norm{\T}\norm{\psi}_{\infty}.
	\end{align*}
	This shows that $(1-t^2)^{\frac{n-2}{2}}\bar{\mu}'(t)\in\M(-1,1)$. Denoting $\mu_0=\mathbbm{1}_{\S^{n-1}\backslash\{\pm\e\}}\mu$, \cref{zonal_reg_order1} and the first part of the proof show that $\T_{\mu_0}$ is a bounded linear operator from $C(\S^{n-1})$ to $C^1(\S^{n-1})$. Hence also
	\begin{equation*}
		\mu(\{\e\})\Id + \mu(\{-\e\})\mathrm{Refl}
		= \T - \T_{\mu_0}
	\end{equation*}
	is a bounded linear operator from $C(\S^{n-1})$ to $C^1(\S^{n-1})$, where $\mathrm{Refl}=\T_{\delta_{-\e}}$ is the reflection at the origin. Clearly, this is possible only if $\mu$ carries no mass at the poles.
\end{proof}


\begin{thm}\label{convol_trafo_C^2}
	If $f\in L^1(\S^{n-1})$ is zonal and $\nabla_{\S}^2f \in\M(\S^{n-1},\R^{n\times n})$, then the convolution transform $\T_f$ is a bounded linear operator from $C(\S^{n-1})$ to $C^2(\S^{n-1})$.
	Conversely, if $\T$ is an $\SO(n)$ equivariant bounded linear operator from $C(\S^{n-1})$ to $C^2(\S^{n-1})$, then there exists a unique zonal $f\in L^1(\S^{n-1})$ satisfying $\nabla_{\S}^2f \in\M(\S^{n-1},\R^{n\times n})$ such that $\T=\T_f$.
	In this case, for every $\phi\in C(\S^{n-1})$ and $u\in\S^{n-1}$,
	\begin{align}
		\nabla_{\S}(\phi\ast f)(u)
		&= \int_{\S^{n-1}} \phi(v)P_{u^\perp}v \bar{f}'(\langle u,v\rangle)dv,\label{eq:convol_trafo_C^2_1}\\
		\nabla_{\S}^2(\phi\ast f)(u)
		&= \int_{\S^{n-1}} \phi(v)(P_{u^\perp}v \otimes P_{u^\perp}v) (\J^u\bar{f}'')(dv) - \int_{\S^{n-1}} \phi(v)\langle u,v\rangle\bar{f}'(\langle u,v\rangle)dv\hspace{.5mm} P_{u^\perp}.\label{eq:convol_trafo_C^2_2}
	\end{align}
\end{thm}
\begin{proof}
	Suppose that $f\!\in\! L^1(\S^{n-1})$ is zonal and that $\nabla_{\S}^2f\! \in\!\M(\S^{n-1},\R^{n\times n})$. Then $\nabla_{\S}f\!\in\! L^1(\S^{n-1},\R^n)$ due to \cref{zonal_reg_order2}. According to \cref{convol_trafo_C^1}, the convolution transform $\T_f$ is a bounded linear operator from $C(\S^{n-1})$ from $C^1(\S^{n-1})$ and identity \cref{eq:convol_trafo_C^2_1} holds for every $\phi\in C(\S^{n-1})$. Next, we will verify identity \cref{eq:convol_trafo_C^2_2} for an arbitrary smooth function $\phi\in C^\infty(\S^{n-1})$. Let $u\in\S^{n-1}$ be a point and $w\in u^\perp$ a tangent vector. Choosing a geodesic $\gamma$ in $\S^{n-1}$ such that $\gamma(0)=u$ and $\gamma'(0)=w$ yields
	\begin{equation*}
		\langle \nabla_{\S}^2(\phi\ast f)(u) , w\otimes w \rangle
		= \frac{d^2}{ds^2}\bigg|_0 (\phi\ast f)(\gamma(s))
		= \frac{d^2}{ds^2}\bigg|_0 \int_{(-1,1)} \J_{\gamma(s)}[\phi](t) \bar{f}(t) dt.
	\end{equation*}
	Note that \cref{zonal_reg_order2} and \cref{beta_int_by_parts} imply that $(1-t^2)^{\frac{n-1}{2}}\bar{f}''(t)$ is a finite signed measure and both $(1-t^2)^{\frac{n-3}{2}}\bar{f}'(t)$ and $(1-t^2)^{\frac{n-4}{2}}\bar{f}(t)$ are integrable functions on $(-1,1)$. Due to \cref{eq:curve_estimate1}, we have that $(1-t^2)^{-\frac{n-4}{2}}\frac{d}{ds}\J_{\gamma(s)}[\phi](t)$ is bounded uniformly in $s$ for all sufficiently small $s$, so we may interchange differentiation and integration and obtain	
	\begin{align*}
		&\langle \nabla_{\S}^2(\phi\ast f)(u) , w\otimes w \rangle
		= \left.\frac{d}{ds}\right|_0\int_{(-1,1)} \frac{d}{ds}\J_{\gamma(s)}[\phi](t) \bar{f}(t) dt\\
		&\qquad = - \left.\frac{d}{ds}\right|_0\int_{(-1,1)} \frac{d}{dt}\J_{\gamma(s)}[\langle \cdot, \gamma'(s)\rangle](t) \bar{f}(t)dt
		= \left.\frac{d}{ds}\right|_0 \int_{(-1,1)} \J_{\gamma(s)}[\langle \cdot,\gamma'(s)\rangle](t) \bar{f}'(t) dt,
	\end{align*}
	where the second equality follows from \cref{eq:curve_deriv1} and the final equality from \cref{eq:beta_int_by_parts}. Due to \cref{eq:curve_estimate2}, we have that $(1-t^2)^{-\frac{n-3}{2}}\frac{d}{ds}\J_{\gamma(s)}[\langle\cdot,\gamma'(s)\rangle\phi](t)$ is uniformly bounded in $s$ for all sufficiently small $s$, so we may again interchange differentiation and integration and obtain
	\begin{align*}
		&\langle \nabla_{\S}^2(\phi\ast f)(u) , w\otimes w \rangle
		=  \int_{(-1,1)} \left.\frac{d}{ds}\right|_0 \J_{\gamma(s)}[\langle \cdot,\gamma'(s)\rangle](t) \bar{f}'(t) dt \\
		&\qquad = \int_{(-1,1)} \left( -\frac{d}{dt}\J_u[\langle\cdot,w\rangle\phi](t) - \abs{w}^2t\J_u[\phi](t) \right)\bar{f}'(t)dt \\
		&\qquad= \int_{(-1,1)} \J_u[(.\cdot w)^2\phi](t)\bar{f}''(dt) - \abs{w}^2\int_{(-1,1)} \J_u[\phi](t)\bar{f}'(t)dt \\
		&\qquad= \int_{\S^{n-1}} \phi(v)\langle P_{u^\perp}v\otimes P_{u^\perp}v , w\otimes w \rangle (\J^u\bar{f}'')(dv) - \int_{\S^{n-1}} \langle u,v\rangle\bar{f}'(\langle u,v\rangle)dv \hspace{.5mm}\langle P_{u^\perp} , w\otimes w\rangle,
	\end{align*}
	where the second equality follows from \cref{eq:curve_deriv2} and \cref{eq:geodesic_representation}, the third from \cref{eq:beta_int_by_parts}, and the final equality from a change to spherical cylinder coordinates. Since the space of $2$-tensors on the tangent space $u^\perp$ is spanned by pure tensors $w\otimes w$, this proves identity \cref{eq:convol_trafo_C^2_2} for all $\phi\in C^\infty(\S^{n-1})$.
	
	As an immediate consequence, $\norm{\T_f \phi}_{C^2(\S^{n-1})} \leq C\norm{\phi}_{C(\S^{n-1})}$ for some constant $C\geq 0$ and all $\phi\in C^\infty(\S^{n-1})$. Thus, $\T_f$ extends to a bounded linear operator $\overline{\T}: C(\S^{n-1})\to C^2(\S^{n-1})$. Since $\T_f:C(\S^{n-1})\to C^1(\S^{n-1})$ is a bounded operator and the inclusion $C^2(\S^{n-1})\subseteq C^1(\S^{n-1})$ is continuous, $\overline{\T}$ agrees with $\T_f$. By density and continuity, \cref{eq:convol_trafo_C^2_2} is valid for all $\phi\in C(\S^{n-1})$.
	
	For the second part of the theorem, suppose that $\T$ is an $\SO(n)$ equivariant bounded linear operator from $C(\S^{n-1})$ to $C^2(\S^{n-1})$. Then \cref{convol_trafo_C^1} implies that $\T=\T_f$ for a unique zonal $f\in L^1(\S^{n-1})$. According to \cref{zonal_reg_order2}, it suffices to show that $(1-t^2)^{\frac{n-1}{2}}\bar{f}''(t)$ is a finite signed measure. To that end, take an arbitrary test function $\psi\in \D(-1,1)$, a point $u\in\S^{n-1}$, a unit tangent vector $w\in u^\perp$ and define
	\begin{equation*}
		\phi(v)
		= P^{n-1}_2\bigg(\frac{\langle v,w\rangle}{\sqrt{1-\langle v,u\rangle^2}}\bigg)\psi(\langle v,u\rangle ),
		\qquad v\in\S^{n-1}.
	\end{equation*}
	Then $\phi$ is a smooth function satisfying $\J_u[\phi](t)=0$ and $\J_u[\langle \cdot,w\rangle^2\phi](t)=C_n(1-t^2)^{\frac{n-1}{2}}\psi(t)$, where $C_n>0$ is given by $C_n=\int_{\S^{n-1}\cap u^\perp} P^{n-1}_2(\langle w,v\rangle)\langle w,v\rangle^2 dv$. Choosing a geodesic $\gamma$ in $\S^{n-1}$ such that $\gamma(0)=u$ and $\gamma'(0)=w$ yields
	\begin{align*}
		&C_n \left< \psi(t), (1-t^2)^{\frac{n-1}{2}}\bar{f}''(t) \right>_{\D'}
		= \left< \J_u[\langle \cdot,w\rangle^2\phi](t), \bar{f}''(t) \right>_{\D'} - \left< \J_u[\phi](t), t\bar{f}'(t) \right>_{\D'} \\
		&\qquad= \int_{(-1,1)} \left(\frac{d^2}{dt^2}\J_u[\langle\cdot,w\rangle^2\phi](t) + \frac{d}{dt}\left(t\J_u[\phi](t)\right)\right)\bar{f}(dt)
		= \int_{(-1,1)} \frac{d^2}{ds^2}\bigg|_0 \J_{\gamma(s)}[\phi](t) \bar{f}(t)dt,
	\end{align*}
	where the final equality is due to \cref{eq:curve_deriv2}. Observe that $\supp \J_{\gamma(s)}[\phi] \subseteq [-1+\varepsilon,1-\varepsilon]$ and that $\abs{\frac{d^2}{ds^2}\J_{\gamma(s)}[\phi](t)}\leq C$ uniformly in $s$ for all sufficiently small $s$, so we may interchange differentiation and integration and obtain
	\begin{equation*}
		C_n \left< \psi(t), (1-t^2)^{\frac{n-1}{2}}\bar{f}''(t) \right>_{\D'}
		= \frac{d^2}{ds^2}\bigg|_0  \int_{(-1,1)}  \J_{\gamma(s)}[\phi](t) \bar{f}(t)dt
		= \langle \nabla_{\S}^2(\phi\ast f)(u) , w\otimes w \rangle.
	\end{equation*}
	Therefore, we arrive at the following estimate:
	\begin{equation*}
		\left| \left< \psi(t), (1-t^2)^{\frac{n-1}{2}}\bar{f}''(t) \right>_{\D'} \right|
		\leq C_n^{-1} \norm{\T} \norm{\phi}_{C(\S^{n-1})}
		\leq C_n^{-1} \norm{\T} \norm{\psi}_{\infty}.
	\end{equation*}
	This shows that $(1-t^2)^{\frac{n-1}{2}}\bar{f}''(t)\in\M(-1,1)$, which completes the proof.
\end{proof}

In Theorems~\ref{convol_trafo_C^1} and \ref{convol_trafo_C^2}, we identify convolution transforms with zonal functions for which the spherical gradient and Hessian are signed measures, respectively. In general, checking these conditions directly can be difficult. However, Propositions~\ref{zonal_reg_order1}, \ref{zonal_reg_order2}, and \ref{zonal_reg_equi} provide more practical equivalent conditions. In this way, we obtain \cref{convol_trafo_C^2:intro}.

\begin{proof}[Proof of \cref{convol_trafo_C^2:intro}]
	According to \cref{convol_trafo_C^2}, the convolution transform $\T_f$ is a bounded linear operator from $C(\S^{n-1})$ to $C^2(\S^{n-1})$ if and only if $\nabla_{\S}^2f \in\M(\S^{n-1},\R^{n\times n})$. Due to \cref{zonal_Laplacian}, this is the case precisely when $\square_nf$ is a finite signed measure and satisfies \cref{eq:convol_trafo_C^2:intro}.
\end{proof}

For a zonal $f\in C^2(\S^{n-1})$, integration and differentiation can be interchanged, and thus,
\begin{equation*}
	\nabla_{\S}^2(\phi\ast f)(u)
	= \int_{\S^{n-1}} \phi(v)\nabla_{\S}^2\bar{f}(\langle\cdot,v\rangle) (u) dv
\end{equation*}
for every $\phi\in C(\S^{n-1})$. In light of \cref{eq:zonal_reg_order2_2}, we see that \cref{eq:convol_trafo_C^2_2} naturally extends this identity to general $f\in L^1(\S^{n-1})$. Denote by $D^2\phi$ the Hessian of the $1$-homogeneous extension of a function $\phi$ on $\S^{n-1}$. Since $D^2\phi(u)=\nabla_{\S}\phi(u)+\phi(u)P_{u^\perp}$, as a direct consequence of \cref{eq:convol_trafo_C^2_2}, we obtain the following formula for:
\begin{align}\label{eq:convol_trafo_C^2_3}
	\begin{split}
		D^2(\phi\ast f)(u)
		&= \int_{\S^{n-1}} \phi(v)(P_{u^\perp}v \otimes P_{u^\perp}v) (\J^u\bar{f}'')(dv)  \\
		&\qquad + \int_{\S^{n-1}} \phi(v)\left(\bar{f}(\langle u,v\rangle) - \langle u,v\rangle\bar{f}'(\langle u,v\rangle)\right)dv\hspace{.5mm}P_{u^\perp}.
	\end{split}		
\end{align}

As an instance of \cref{convol_trafo_C^2}, we obtain Martinez-Maure's \cite{MR1814896} result on the cosine transform, which is discussed in the following example.

\begin{exl}
	The \emph{cosine transform} is the convolution transform $\T_f$ generated by the $L^1(\S^{n-1})$ function $f(u)=\abs{\langle\e,u\rangle}$, that is, $\bar{f}(t)=\abs{t}$. Thus $\bar{f}''=2\delta_0$ in the sense of distributions, so \cref{zonal_reg_order2} and \cref{convol_trafo_C^2} imply that the cosine transform is a bounded linear operator from $C(\S^{n-1})$ to $C^2(\S^{n-1})$.

	Moreover, $\bar{f}(t)-t\bar{f}'(t)=0$ and $\J^u\bar{f}''=2\lambda_{\S^{n-1}\cap u^\perp}$, where $\lambda_{\S^{n-1}\cap u^\perp}$ denotes the Lebesgue measure on the $(n-2)$-dimensional subsphere $\S^{n-1}\cap u^\perp$. Hence \cref{eq:convol_trafo_C^2_3} shows that for every $\phi\in C(\S^{n-1})$,	
	\begin{equation*}
		D^2(\phi\ast f)(u)
		= 2 \int_{\S^{n-1}\cap u^\perp} \phi(v)(P_{u^\perp}v \otimes P_{u^\perp}v)dv
		= 2 \int_{\S^{n-1}\cap u^\perp} \phi(v)(v\otimes v)dv.
	\end{equation*}
\end{exl}

\begin{exl}\label{exl:Berg_fct_convol_trafo}
	For Berg's function $g_n$, we have seen in \cref{exl:Berg_fct_regularity} that $\nabla_{\S}\breve{g}_n$ is an integrable function on $\S^{n-1}$ while the distributional spherical Hessian $\nabla_{\S}^2\breve{g}_n$ is not a finite signed measure. Thus, Theorems~\ref{convol_trafo_C^1} and \ref{convol_trafo_C^2} imply that the convolution transform $\T_{\breve{g}_n}$ is a bounded operator from $C(\S^{n-1})$ to $C^1(\S^{n-1})$ but not a bounded operator from $C(\S^{n-1})$ to $C^2(\S^{n-1})$.
\end{exl}

\section{Regularity of Minkowski Valuations}
\label{sec:w-mon}

In this section, we study the regularity of Minkowski valuations $\Phi_i\in\MVal_i$ of degrees $1\leq i\leq n-1$, proving \cref{w_mon=>w_pos+density:intro}. In the $(n-1)$-homogeneous case, Schuster~\cite{MR2327043} showed that $\Phi_{n-1}$ is generated by a continuous function. For other degrees of homogeneity, all that is known about the regularity of a generating function $f$ is that $\square_n f$ is a signed measure and $f$ is integrable (due to Dorrek~\cite{MR3655954}). Using our study of regularity of zonal functions in \cref{sec:convol}, we are able to refine Dorrek's results.

\begin{thm}\label{gen_fct_reg}
	Let $1\leq i\leq n-1$ and $\Phi_i\in\MVal_i$ with generating function $f$. Then
	\begin{enumerate}[label=\upshape(\roman*),topsep=0.5ex,itemsep=-0.5ex]
		\item \label{gen_fct_reg:box_f}
		$\square_nf$ is a signed measure on $\S^{n-1}$,
		\item \label{gen_fct_reg:loc_Lipschitz}
		$f$ is a locally Lipschitz function on $\S^{n-1}\backslash\{\pm\e\}$,
		\item \label{gen_fct_reg:diff}
		$f$ is differentiable almost everywhere on $\S^{n-1}$, and $\nabla_{\S}f\in L^1(\S^{n-1},\R^n)$.
	\end{enumerate}
\end{thm}
\begin{proof}
	By \cite[Theorem~6.1(i)]{MR3432270}, the function $f$ also generates a Minkowski valuation of degree one. It follows from \cite[Theorem~1.2]{MR3655954} that $\square_nf$, and thus $\Delta_{\S}f$, is a signed measure on $\S^{n-1}$.
	Therefore, \cref{eq:zonal_Laplacian} shows that $(1-t^2)^{\frac{n-1}{2}}\bar{f}'(t)$ is an $L^\infty(-1,1)$ function. This implies that $\bar{f}$ is locally Lipschitz on $(-1,1)$, and thus, $f$ is locally Lipschitz on $\S^{n-1}\backslash\{\pm\e\}$. Moreover, $(1-t^2)^{\frac{n-2}{2}}\bar{f}'(t)$ is in $L^1(-1,1)$, so due to \cref{zonal_reg_order1}, the distributional gradient $\nabla_{\S}f$ is in $L^1(\S^{n-1},\R^n)$. Since $f$ is locally Lipschitz on $\S^{n-1}\backslash\{\pm\e\}$, according to Rademacher's theorem, the classical gradient of $f$ exists almost everywhere on $\S^{n-1}$ and agrees with the distributional gradient.
\end{proof}

As a consequence of Theorems~\ref{convol_trafo_C^1} and \ref{gen_fct_reg}, we obtain the following.

\begin{cor}
	For $1\leq i\leq n-1$, every Minkowski valuation $\Phi_i\in\MVal_i$ maps convex bodies with a $C^2$ support function to strictly convex bodies.
\end{cor}
\begin{proof}
	Denote by $f$ the generating function of $\Phi_i$. \cref{gen_fct_reg}~\ref{gen_fct_reg:diff} implies that $\nabla_{\S}f\in L^1(\S^{n-1},\R^n)$. According to \cref{convol_trafo_C^1}, the convolution transform $\T_f$ is a bounded operator from $C(\S^{n-1})$ to $C^1(\S^{n-1})$. Suppose now that $K\in\K^n$ has a $C^2$ support function. Then $S_i(K,\cdot)$ has a continuous density, so $h(\Phi_iK,\cdot)=S_i(K,\cdot)\ast f$ is a $C^1(\S^{n-1})$ function, and thus, $\Phi_iK$ is strictly convex (see, e.g., \cite[Section~2.5]{MR3155183}).
\end{proof}

We now turn to weakly monotone Minkowski valuations, for which we will obtain additional regularity of their generating functions. Recall that $\Phi_i:\K^n\to\K^n$ is called \emph{weakly monotone} if $\Phi_i K\subseteq \Phi_i L$ whenever $K\subseteq L$ and the Steiner points of $K$ and $L$ are at the origin. 

\begin{thm}\label{w_mon=>w_pos+density}
	Let $1\leq i\leq n-1$ and $\Phi_i\in\MVal_i$ be weakly monotone with generating function $f$. Then $\square_nf$ is a weakly positive measure on $\S^{n-1}$ and there exists $C> 0$ such that for all $r\geq 0$,
	\begin{equation}\label{eq:w_mon=>w_pos+density}
		\abs{\square_nf}\big( \{u\in\S^{n-1}:\abs{\langle\e,u\rangle}> \cos r \} \big)
		\leq Cr^{i-1}.
	\end{equation}
\end{thm}

Note that Theorems~\ref{gen_fct_reg} and \ref{w_mon=>w_pos+density} together yield \cref{w_mon=>w_pos+density:intro}.
Here and in the following, a distribution on $\S^{n-1}$ is called \emph{weakly positive} if it can be written as the sum of a positive measure and a linear function. In particular, every weakly positive distribution is a signed measure.
The following characterization of weak positivity is a simple consequence of the Hahn-Banach separation theorem. For completeness, we provide a proof in \cref{sec:omitted}.

\begin{lem}\label{w_pos_iff}
	A distribution $\nu\in C^{-\infty}(\mathbb{S}^{n-1})$ is weakly positive if and only if $\left< \phi, \nu \right>_{C^{-\infty}} \geq 0$ for every positive centered smooth function $\phi\in C^\infty(\mathbb{S}^{n-1})$.
\end{lem}

The proof of \cref{w_mon=>w_pos+density} relies on the behavior of area measures of convex bodies on spherical caps. We need the following classical result by Firey (recall the notation introduced in \cref{eq:SCap}).

\begin{thm}[\cite{MR286982}]\label{area_meas_density}
	Let $1\leq i\leq n-1$ and $K\in\K^n$ be a convex body. Then for every $u\in\S^{n-1}$,
	\begin{equation}\label{eq:area_meas_density}
		S_i(K,\SCap_r(u))
		\leq C_{n,i}(\diam K)^i r^{n-1-i}
	\end{equation}
	where $C_{n,i}>0$ depends only on $n$ and $i$ and $\diam K$ denotes the diameter of $K$.
\end{thm}

The area measures of the $(n-1)$-dimensional disk in $\e^\perp$, which we denote by $D^{n-1}$, exhibit the worst possible asymptotic behavior in \cref{eq:area_meas_density}. This is shown in the example below.

\begin{exl}\label{exl:disk_area_meas}
	We seek to compute the area measures of $D^{n-1}$. For a convex body $K\in\K^n$ it is well known that $S_{n-1}(K,A)$ is the area of the reverse spherical image of a measurable subset $A\subseteq\S^{n-1}$. Thus,
	\begin{equation}\label{eq:disk_area_meas_i=n-1}
		S_{n-1}(D^{n-1},\cdot)
		= \kappa_{n-1}\left(\delta_{-\e}+\delta_{\e}\right).
	\end{equation}
	In order to compute the area measures of lower order, note that if a convex body $K\in\K^n$ with absolutely continuous area measure of order $1\leq i< n-1$ lies in a hyperplane $u^\perp$ (where $u\in\S^{n-1}$), then
	\begin{equation*}
		s_i(K,v)
		= \frac{n-1-i}{n-1}\left(1-\langle u,v\rangle^2\right)^{-\frac{i}{2}} s_i^{u^\perp}\left(K,\frac{P_{u^\perp}v}{\abs{P_{u^\perp}v}}\right),
		\qquad v\in\S^{n-1}\backslash\{\pm u\},
	\end{equation*}
	where $s_i(K,\cdot)$ and $s_i^{u^\perp}(K,\cdot)$ denote the densities of the $i$-th area measure of $K$ with respect to the ambient space and the hyperplane $u^\perp$, respectively (see \cite[Lemma~3.15]{KiderlenPhD}). Hence, for $1\leq i< n-1$,
	\begin{equation*}
		S_i(D^{n-1},dv)
		= \frac{n-1-i}{n-1}\left(1-\langle \e,v\rangle^2\right)^{-\frac{i}{2}}dv.
	\end{equation*}
	By a change to spherical cylinder coordinates, the measure of a polar cap can be estimated by
	\begin{equation}\label{eq:disk_area_meas_i<n-1}
		S_i(D^{n-1},\SCap_r(\e))
		\geq \frac{n-1-i}{n-1}\omega_{n-1}\int_{[\cos r,1]} t(1-t^2)^{\frac{n-3-i}{2}} dt 
		= \kappa_{n-1}(\sin r)^{n-1-i}.
	\end{equation}
\end{exl}

We require the following simple lemma, which relates the behavior of two positive measures $\mu$ and $\nu$ on small polar caps to the behavior of their convolution product.

\begin{lem}\label{caps_convol}
	Let $\mu\in\M_+(\S^{n-1})$ and let $\nu\in\M_+(\S^{n-1})$ be zonal. Then for all $u\in\S^{n-1}$ and $r\geq 0$,
	\begin{align}
		\mu(\SCap_r(u))\nu(\SCap_r(\e))
		&\leq (\mu\ast\nu)(\SCap_{2r}(u)), \label{eq:caps_convol_+}\\
		\mu(\SCap_r(u))\nu(\SCap_r(-\e))
		&\leq (\mu\ast\nu)(\SCap_{2r}(-u)). \label{eq:caps_convol_-}
	\end{align}
\end{lem}
\begin{proof}
	First observe that \cref{eq:caps_convol_+} implies \cref{eq:caps_convol_-} by reflecting $\nu$ at the origin. Moreover, we may assume that $u=\e$. The general case can be obtained from this by applying a suitable rotation to the measure $\mu$ and exploiting the $\SO(n)$ equivariance of the convolution transform $\T_\nu$. Next, note that
	\begin{equation*}
		(\mu\ast\nu)(\SCap_{2r}(\e))
		= \int_{\S^{n-1}} \mu(\SCap_{2r}(u)) \nu(du),
	\end{equation*}
	as can be easily shown by approximating $\mathbbm{1}_{\SCap_{2r}(\e)}$ with smooth functions from below and applying the principle of monotone convergence. Thus
	\begin{equation*}
		(\mu\ast\nu)(\SCap_{2r}(\e))
		\geq \int_{\SCap_r(\e)} \mu(\SCap_{2r}(u)) \nu(du)
		\geq \int_{\SCap_r(\e)} \mu(\SCap_{r}(\e)) \nu(du)
		= \mu(\SCap_r(\e))\nu(\SCap_r(\e)),
	\end{equation*}
	where the first inequality follows from shrinking the domain of integration and the second from the fact that $\SCap_{2r}(u)\supseteq \SCap_r(\e)$ for all $u\in\SCap_r(\e)$ combined with the monotonicity of $\mu$.
\end{proof}

Now we are in a position to prove \cref{w_mon=>w_pos+density}.

\begin{proof}[Proof of \cref{w_mon=>w_pos+density}]
	Define a functional $\xi_i$ on the space of smooth support functions by $\xi_i(h_K)=h(\Phi_i K,\e)$. For every $\phi\in C^\infty(\S^{n-1})$, the function $1+t\phi$ is a support function whenever $t\in\R$ is sufficiently small. Therefore, we may compute the first variation of $\xi_i$ at $\phi_0=1$. To that end, note that as a consequence of \cref{eq:MVal_representation:intro} and the polynomiality of area measures (see, e.g., \cite[Section~5.1]{MR3155183}),
	\begin{equation*}
		\xi_i(1+t\phi)
		= a^n_0[f] + i\left< \square_n \phi,f\right>_{C^{-\infty}} t + O(t^2)
		\qquad \text{as }t\to 0.
	\end{equation*}
	Thus, for the first variation we obtain
	\begin{equation*}
		\delta\xi_i(1,\phi)
		= \left.\frac{d}{dt}\right|_0 \xi_i(1+t\phi)
		= i \left< \square_n \phi,f\right>_{C^{-\infty}}
		= i \left< \phi,\square_n f\right>_{C^{-\infty}}.
	\end{equation*}
	Since $\Phi_i$ is weakly monotone, the functional $\xi_i$ is monotone on the subspace of centered functions, that is, $\xi_i(\phi_1)\leq\xi_i(\phi_2)$ whenever $\phi_1$ and $\phi_2$ are smooth centered support functions and $\phi_1\leq\phi_2$. Consequently, the first variation $\delta\xi_i(1,\phi)$ must be non-negative for every positive and centered $\phi\in C^\infty(\S^{n-1})$. \cref{w_pos_iff} implies that $\square_nf$ is a weakly positive measure.
	
	Since $\square_nf$ is weakly positive, there exists some positive measure $\nu\in\M_+(\S^{n-1})$ and $x\in\R^n$ such that $\square_nf=\nu+\langle x,\cdot\rangle$. Observe that for every $K\in\K^n$, 
	\begin{equation*}
		S_1(\Phi_i K,\cdot)
		= \square_n h(\Phi_iK,\cdot)
		= \square_n (S_i(K,\cdot)\ast f)
		= S_i(K,\cdot) \ast \square_n f
		= S_i(K,\cdot) \ast \nu.
	\end{equation*}
	Hence, \cref{eq:caps_convol_+} and \cref{eq:caps_convol_-} imply that for all $r\geq 0$,
	\begin{equation*}
		S_i(K,\SCap_r(\e))\nu(\SCap_r(\pm\e))
		\leq S_1(\Phi_iK,\SCap_{2r}(\pm\e)).
	\end{equation*}
	Due to \cref{eq:area_meas_density}, the right hand side is bounded from above by a constant multiple of $r^{n-2}$. If we choose $K$ to be the $(n-1)$-dimensional disk $D^{n-1}$, then $S_i(K,\SCap_r(\e))$ is bounded from below by a multiple of $r^{n-i-1}$, as is shown in \cref{eq:disk_area_meas_i=n-1} and \cref{eq:disk_area_meas_i<n-1}. Thus,
	\begin{equation*}
		\nu(\SCap_{r}(\pm\e))
		\leq \frac{S_1(\Phi_iD^{n-1},\SCap_{2r}(\pm\e))}{S_i(D^{n-1},\SCap_r(\e))}
		\leq C' \frac{r^{n-2}}{r^{n-i-1}}
		= C'r^{i-1}.
	\end{equation*}
	Since $\abs{\square_nf}\leq \nu + \abs{\langle x,\cdot\rangle}$, we have that
	\begin{equation*}
		\abs{\square_nf}(\SCap_r(\pm\e))
		\leq \nu(\SCap_r(\pm\e)) + \int_{\SCap_r(\pm\e)} \abs{\langle x,u\rangle} du
		\leq C'r^{i-1} + \abs{x}\kappa_{n-1}r^{n-1}
		\leq Cr^{i-1}
	\end{equation*}	
	for a suitable constant $C\geq 0$, which proves \cref{eq:w_mon=>w_pos+density}.
\end{proof}

By combining \cref{w_mon=>w_pos+density} with \cref{convol_trafo_C^2:intro}, we immediately obtain the following.

\begin{cor}\label{w_mon=>bounded_C^2}
	Let $1<i\leq n-1$ and $\Phi_i\in\MVal_i$ be weakly monotone with generating function $f$. \linebreak Then the convolution transform $\T_f$ is a bounded linear operator from $C(\S^{n-1})$ to $C^2(\S^{n-1})$.	
\end{cor}

As was pointed out in \cref{zonal_reg_equi}, the behavior of zonal measures on small polar caps determines their regularity. In the following, we show that this behavior also determines the rate of convergence of their multipliers, which is another way of expressing regularity. We use the following classical asymptotic estimate for Legendre polynomials.

\begin{thm}[\!\! {\cite[7.33]{MR0372517}}] \label{Legendre_polynomial_estimate}
	For all $n\geq 3$ and $\delta>0$, there exists $M>0$ such that for all $k\geq 0$,
	\begin{equation} \label{eq:Legendre_polynomial_estimate}
		\abs{P^n_k(t)}
		\leq M k^{-\frac{n-2}{2}}(1-t^2)^{-\frac{n-2}{4}}
		\qquad\text{for } t\in\left[-\cos{\tfrac{\delta}{k}},\cos{\tfrac{\delta}{k}}\right].
	\end{equation}
\end{thm}

\begin{thm}\label{density=>asymptotics}
	Let $\mu\in\M(\S^{n-1})$ be zonal and suppose that there exist $C> 0$ and $\alpha\geq 0$ such that 
	\begin{equation*}
		\abs{\mu}\big(\{u\in\S^{n-1}: \abs{\langle \e,u\rangle}>\cos r\}\big)
		\leq Cr^\alpha
	\end{equation*}
	for all $r\geq 0$. Then
	\begin{equation*}
		a^n_k[\mu]
		\in \left\{\begin{array}{ll}
			O(k^{-\alpha}),					&\alpha<\frac{n-2}{2},\\
			O(k^{-\frac{n-2}{2}}\ln k),		&\alpha=\frac{n-2}{2},\\
			O(k^{-\frac{n-2}{2}}),			&\alpha>\frac{n-2}{2}.
		\end{array}\right.
	\end{equation*}
\end{thm}
\begin{proof}
	Since $\mu$ can be decomposed into two signed measures that are each supported on one hemisphere, we may assume that $\supp\mu\subseteq\{u\in\S^{n-1}:\langle\e,u\rangle\geq 0\}$. Denoting by $\rho=\J_{\e}[\abs{\mu}]$ the pushforward measure of $\abs{\mu}$ with respect to the map $u\mapsto\langle\e,u\rangle$, we have that
	\begin{equation*}
		\abs{a^n_k[\mu]}
		= \left| \int_{\S^{n-1}} P^n_k(\langle\e,u\rangle)\mu(du)\right|
		\leq \int_{\S^{n-1}} \abs{P^n_k(\langle\e,u\rangle)} ~\abs{\mu}(du)
		= \int_{[0,1]} \abs{P^n_k(t) } \rho(dt).
	\end{equation*}
	Our aim now is to find suitable bounds for the integral on the right hand side. To that end, we fix some arbitrary $\delta>0$ and split it into the two integrals 
	\begin{equation*}
		I_1(k)
		= \int_{[0,\cos{\frac{\delta}{k}}]} \abs{P^n_k(t)}\rho(dt)
		\qquad\text{and}\qquad
		I_2(k)
		= \int_{(\cos{\frac{\delta}{k}},1]} \abs{P^n_k(t)}\rho(dt).
	\end{equation*}
	Observe that our assumption on $\mu$ implies that $\rho((t,1])\leq C(1-t^2)^{\frac{\alpha}{2}}$ for all $t\in[0,1]$. Since $\abs{P^n_k(t)}\leq 1$ for all $t\in [0,1]$, we obtain
	\begin{equation*}
		I_2(k)
		\leq \rho\left( \left(\cos{\tfrac{\delta}{k}},1\right] \right)
		\leq C \left(\sin{\tfrac{\delta}{k}}\right)^{\alpha}
		\in O(k^{-\alpha}).
	\end{equation*}
	For the integral $I_1(k)$, estimate \cref{eq:Legendre_polynomial_estimate} and Lebesgue-Stieltjes integration by parts yield
	\begin{align*}
		I_1(k)
		&\leq Mk^{-\frac{n-2}{2}}\int_{[0,\cos{\frac{\delta}{k}}]} (1-t^2)^{-\frac{n-2}{4}}\rho(dt) \\
		&= Mk^{-\frac{n-2}{2}} \left(\rho\left(\left[0,\cos{\tfrac{\delta}{k}}\right]\right) -  \left(\sin{\tfrac{\delta}{k}}\right)^{-\frac{n-2}{2}} \rho\left(\left(\cos{\tfrac{\delta}{k}},1\right]\right) + \tilde{I}_1(k)\right) \\
		&\leq Mk^{-\frac{n-2}{2}}\left(\rho([0,1])+\tilde{I}_1(k)\right),
	\end{align*}
	where we defined
	\begin{equation*}
		\tilde{I}_1(k)
		= \tfrac{n-2}{2}\int_{[0,\cos{\frac{\delta}{k}}]} \rho\left(\left(t,1\right]\right) t(1-t^2)^{-\frac{n+2}{4}}dt.
	\end{equation*}
	Employing again our estimate on $\rho((t,1])$ and performing a simple computation shows that
	\begin{equation*}
		\tilde{I}_1(k)
		\leq C\tfrac{n-2}{2} \int_{[0,{\cos\frac{\delta}{k}}]} t(1-t^2)^{\frac{1}{2}\left(\alpha-\frac{n-2}{2}-1\right)} dt
		\in \left\{\begin{array}{ll}
			O(k^{\frac{n-2}{2}-\alpha}),	&\alpha<\frac{n-2}{2},\\
			O(\ln k),						&\alpha=\frac{n-2}{2},\\
			O(1),							&\alpha>\frac{n-2}{2}.
		\end{array}\right.
	\end{equation*}
	Combining the estimates for $I_1(k)$ and $I_2(k)$ completes the proof.
\end{proof}

As an immediate consequence of Theorems~\ref{w_mon=>w_pos+density} and \ref{density=>asymptotics}, we obtain the following.

\begin{cor}\label{w_mon=>decay}
	Let $1< i\leq n-1$ and $\Phi_i\in\MVal_i$ be weakly monotone with generating function $f$. \linebreak Then $a^n_k[\square_n f]\in O(k^{-1/2})$ as $k\to \infty$.
\end{cor}

\section{Fixed Points}
\label{sec:fixed_points}

In this section, we prove a range of results regarding local uniqueness of fixed points of Minkowski valuations $\Phi_i\in\MVal_i$ of degree $1<i\leq n-1$ (the $1$-homogeneous case has been settled globally by Kiderlen~\cite{MR2238926}). This section is divided into three subsections. In \cref{sec:MSO}, we prove \cref{fixed_points_MSO:intro} concerning the mean section operators. \cref{sec:even_i=n-1} is dedicated to Minkowski valuations $\Phi_i$ generated by origin-symmetric convex bodies of revolution. There we prove \cref{fixed_points_body_of_rev:intro}, unifying previous results by Ivaki \cite{MR3639524,MR3757054} and the second author and Schuster \cite{MR4316669}. Finally, in \cref{sec:even_i<n-1} we consider general even Minkowski valuations for which we obtain information about the fixed points of $\Phi_i$ (as opposed to $\Phi_i^2$).

The proofs given in this section utilize the following result for general Minkowski valuations $\Phi_i\in\MVal_i$. It provides three sufficient conditions on the generating function of $\Phi_i$ to obtain the desired local uniqueness of fixed points of $\Phi_i^2$. It contains however no information on when these conditions are fulfilled. For instance, in the particular case when $\Phi_i$ is generated by an origin symmetric $C^2_+$ convex body of revolution, checking condition  \ref{fixed_points_Phi^2:inj} turns out to be rather involved.

\begin{thm}[{\cite{MR4316669}}]\label{fixed_points_Phi^2}
	Let $1<i\leq n-1$ and $\Phi_i\in\MVal_i$ with generating function $f$ satisfying the following conditions:
	\begin{enumerate}[label=\upshape(C\arabic{*}),topsep=0.5ex,itemsep=-0.5ex,leftmargin=12mm]
		\item \label{fixed_points_Phi^2:diff} the convolution transform $\T_f$ is a bounded linear operator from $C(\S^{n-1})$ to $C^2(\S^{n-1})$,
		\item \label{fixed_points_Phi^2:surj} there exists $\alpha>0$ such that $a^n_k[\square_nf]\in O(k^{-\alpha})$ as $k\to\infty$,
		\item \label{fixed_points_Phi^2:inj} for all $k\geq 2$,
		\begin{equation*}\label{eq:fixed_points_Phi^2:inj}
			\frac{\abs{a^n_k[\square_nf]}}{a^n_0[\square_nf]}
			< \frac{1}{i}.
		\end{equation*}		
	\end{enumerate}
	Then there exists a $C^2$ neighborhood of $B^n$ where the only fixed points of $\Phi_i^2$ are Euclidean balls.
\end{thm}

Now we can apply our results on regularity of weakly monotone Minkowski valuations $\Phi_i$ to these fixed point problems. Corollaries~\ref{w_mon=>bounded_C^2} and \ref{w_mon=>decay} show that conditions \ref{fixed_points_Phi^2:diff} and \ref{fixed_points_Phi^2:surj} are fulfilled in the weakly monotone case, which yields the following.

\begin{thm}\label{w_mon=>fixed_points_Phi^2}
	Let $1< i\leq n-1$ and $\Phi_i\in\MVal_i$ be weakly monotone with generating function $f$ satisfying condition \ref{fixed_points_Phi^2:inj}. Then there exists a $C^2$ neighborhood of $B^n$ where the only fixed points of $\Phi_i^2$	are Euclidean balls.
\end{thm}

\begin{rem}
	The way \cref{fixed_points_Phi^2} was stated in \cite{MR4316669} additionally required $\Phi_i$ to be \emph{even} as the proof employs the following classical result by Strichartz \cite{MR782573}. Denote by $H^s(\S^{n-1})$, $s\in\mathbb{N}$, the Sobolev space of functions on $\S^{n-1}$ with weak covariant derivatives up to order $s$ in $L^2(\S^{n-1})$. Strichartz showed that
	\begin{equation}\label{eq:Sobolev_norm}
		\norm{\phi}_{H^s}^2
		\approx \sum_{k=0}^\infty (k^2+1)^s\norm{\pi_k \phi}_{L^2}^2
	\end{equation}
	for every \emph{even} $\phi\in H^s(\S^{n-1})$, where $\norm{\cdot}_{H^s}$ is the standard norm of $H^s(\S^{n-1})$. However, the classical theory on the Dirichlet problem on compact Riemannian manifolds (see, e.g., \cite[Section~5.1]{MR2744150}) implies that \cref{eq:Sobolev_norm} holds for \emph{every} $\phi\in H^s(\S^{n-1})$. Therefore, by a minor modification of the proof of \cite[Theorem~6.1]{MR4316669}, the assumption on $\Phi_i$ to be even can be omitted.
\end{rem}

\subsection{Mean Section Operators}
\label{sec:MSO}

As a first application, we show local uniqueness of fixed points of the mean section operators $\MSO_j$, which were defined at the beginning of this article. As was pointed out in the introduction, the mean section operators are not generated by a convex body of revolution. This is the main reason why they have not been included in previous results. Due to our extensive study of regularity, we obtain \cref{fixed_points_MSO:intro} as a simple consequence of \cref{w_mon=>fixed_points_Phi^2}.

\setcounter{thmB}{1}
\begin{thmB}\label{fixed_points_MSO}
	For $2 \leq j <n$, there exists a $C^2$ neighborhood of $B^n$ where the only fixed points of $\MSO_j^2$ are Euclidean balls.
\end{thmB}
\begin{proof}
	Define the $j$-th \emph{centered} mean section operator by $\tilde{\MSO}_jK = \MSO_j(K-s(K))$. Then $\tilde{\MSO}_j\in\MVal_i$ for $i=n+1-j$ and due to \cref{eq:MSO_representation} its generating function is given by $(\Id-\pi_1)\breve{g}_j$. Clearly $\MSO_j$ is monotone, and thus, $\tilde{\MSO}_j$ is weakly monotone. By \cref{w_mon=>fixed_points_Phi^2}, it suffices to check condition \ref{fixed_points_Phi^2:inj} for $\breve{g}_j$.
	
	It was shown in \cite{MR3787383} and \cite{MR3769988} independently that the multipliers of $\breve{g}_j$ are given by
	\begin{equation*}
		a^n_k[\breve{g}_j]
		= - \frac{\pi^{\frac{n-j}{2}}(j-1)}{4} \frac{\Gamma\big(\frac{n-j+2}{2}\big)\Gamma\big(\frac{k-1}{2}\big)\Gamma\big(\frac{k+j-1}{2}\big)}{\Gamma\big(\frac{k+n-j+1}{2}\big)\Gamma\big(\frac{k+n+1}{2}\big)}
	\end{equation*}
	for $k\neq 1$. A simple computation using \cref{eq:box_eigenvalues} and the functional equation $\Gamma(x+1)=x\Gamma(x)$ yields
	\begin{equation*}
		\frac{a^n_k[\square_n\breve{g}_j]}{a^n_0[\square_n\breve{g}_j]}
		= \frac{1}{i} \frac{\Gamma\big(\frac{n-1}{2}\big)\Gamma\big(\frac{i+2}{2}\big)\Gamma\big(\frac{k+1}{2}\big)\Gamma\big(\frac{k+n-i}{2}\big)}{\Gamma\big(\frac{n-i}{2}\big)\Gamma\big(\frac{3}{2}\big)\Gamma\big(\frac{k+i}{2}\big)\Gamma\big(\frac{k+n-1}{2}\big)}.
	\end{equation*}
	Since the Gamma function is strictly positive and strictly increasing on $[\frac{3}{2},\infty)$, it follows that $\breve{g}_j$ satisfies condition \ref{fixed_points_Phi^2:inj}.
\end{proof}

\subsection{Convex Bodies of Revolution}
\label{sec:even_i=n-1}

We now turn to Minkowski valuations that are generated by a convex body of revolution, that is, their generating function is a support function. This class includes all even Minkowski valuations in $\MVal_{n-1}$, as was shown in \cite{MR2327043}. Our aim for this section is to prove \cref{fixed_points_body_of_rev:intro} which is restated below.

\setcounter{thmA}{0}
\begin{thmA}\label{fixed_points_body_of_rev}
	Let $1<i\leq n-1$ and $\Phi_i\in\MVal_i$ be generated by an origin-symmetric convex body of revolution. Then there exists a $C^2$ neighborhood of $B^n$ where the only fixed points of $\Phi_i^2$ are Euclidean balls, unless $\Phi_i$ is a multiple of the projection body operator, in which case ellipsoids are also fixed points.
\end{thmA}

Note that if $\Phi_i\in\MVal_i$ is generated by a convex body of revolution $L$, then for every $K\in\K^n$,
\begin{equation*}
	h(\Phi_i K,u)
	= V(K[i],L(u),B^n[n-i-1]),
	\qquad u\in\S^{n-1},
\end{equation*}
where $L(u)$ denotes a suitably rotated copy of $L$, and $V(K_1,\ldots,K_n)$ is the mixed volume of the convex bodies $K_1,\ldots,K_n\in\K^n$ (see, e.g., \cite[Section~5.1]{MR3155183}). Due to the monotonicity of the mixed volume we see that every Minkowski valuation generated by a convex body of revolution is monotone. In light of \cref{w_mon=>fixed_points_Phi^2}, it is natural to ask when condition \ref{fixed_points_Phi^2:inj} is fulfilled. The following result shows that $L$ being origin symmetric is already sufficient up to the second multiplier.

\begin{thm}[{\cite{MR4316669}}]\label{conv_ineq}
	Let $L$ be a convex body of revolution. Then for all even $k\geq 4$,
	\begin{equation}\label{eq:conv_ineq_k>2}
		\frac{\abs{a^n_k[\square_n h_L]}}{a^n_0[\square_n h_L]}
		< \frac{1}{n-1},
	\end{equation}
	and
	\begin{equation}\label{eq:conv_ineq_k=2}
		- \frac{1}{n-1}
		\leq \frac{a^n_2[\square_n h_L]}{a^n_0[\square_n h_L]}
		< \frac{1}{n-1},
	\end{equation}
	where the left hand side inequality in \cref{eq:conv_ineq_k=2} is strict if $L$ is of class $C^2_+$.
\end{thm}

If $i<n-1$, then condition \ref{fixed_points_Phi^2:inj} is fulfilled. If $i=n-1$, then \cref{conv_ineq} shows that condition \ref{fixed_points_Phi^2:inj} is fulfilled under the additional assumption that $L$ is of class $C^2_+$. We will show that imposing this regularity is not necessary: line segments are the only bodies for which equality is attained in the left hand side of inequality \cref{eq:conv_ineq_k=2}.

\begin{defi}\label{defi:EV}
	On $(-1,1)$, we define the two differential operators
	\begin{equation}\label{eq:defi:EV}
		\A_1 
		= \Id  - t\frac{d}{dt}
		\qquad\text{and}\qquad
		\A_2 
		= (1-t^2)\frac{d^2}{dt^2} + \Id - t \frac{d}{dt}.
	\end{equation}
\end{defi}

These operators come up naturally in the study of zonal functions. For a zonal function \linebreak $f\in C^2(\S^{n-1})$, the Hessian of its $1$-homogeneous extension $D^2f$ at each point only has two eigenvalues: $\A_1 \bar{f}$ is the eigenvalue of multiplicity $n-2$, and $\A_2 \bar{f}$ the eigenvalue of multiplicity one (see~\cref{eq:convol_trafo_C^2_3}). The following lemma is a simple consequence of this fact.

\begin{lem}[{\cite{MR4316669}}] \label{support_fct_EV}
	Let $g\in C[-1,1]$. Then $g(\langle\e,\cdot\rangle)$ is the support function of a convex body of revolution if and only if $\A_1g\geq 0$ and $\A_2g\geq 0$ in the weak sense.
\end{lem}

In \cite{MR4316669}, this lemma was proven only for $C^2[-1,1]$ functions, however it extends to $C[-1,1]$ by a simple approximation argument. Next, we determine the kernels of $\A_1$ and $\A_2$.

\begin{lem}\label{kernel_EV}
	Let $g$ be a locally integrable function on $(-1,1)$.
	\begin{enumerate}[label=\upshape(\roman*),topsep=0.5ex,itemsep=-0.5ex]
		\item \label{kernel_EV1} $\A_1 g = 0$ in the weak sense if and only if $g(t)=c_1\abs{t} + c_2t$ for some $c_1, c_2\in\R$.
		\item \label{kernel_EV2} $\A_2 g = 0$ in the weak sense if and only if $g(t)=c_1\sqrt{1-t^2} + c_2t$ for some $c_1,c_2\in\R$.
	\end{enumerate}
\end{lem}
\begin{proof}
	Clearly, $g(t)=c_1\abs{t}+c_2t$ is a weak solution of the differential equation $\A_1g=0$. Conversely, suppose that $\A_1g=0$ in the weak sense. Observe that on the interval $(0,1)$, the change of variables $t=e^s$ transforms the differential operator $\A_1$ as follows:
	\begin{equation*}
		\tilde{\A}_1
		= \Id - \frac{d}{ds}.
	\end{equation*}
	Therefore, there exists $c_+\in\R$ such that $g(e^s)=c_+ e^s$ in $\D'(-\infty,0)$. Reversing the change of variables yields $g(t)=c_+ t$ in $\D'(0,1)$. Similarly, there exists $c_-\in\R$ such that $g(t)=c_-t$ in $\D'(-1,0)$. Since $g$ is locally integrable, choosing $c_1=\frac{1}{2}(c_+ + c_-)$ and $c_2=\frac{1}{2}(c_+-c_-)$, we obtain that $g(t)=c_1\abs{t}+c_2t$ in $\D'(-1,1)$.
	
	For the second part of the lemma, note that $g(t)=c_1\sqrt{1-t^2} + c_2t$ solves the differential equation $\A_2 g=0$. Conversely, suppose that $\A_2 g=0$ in the weak sense. Observe that the change of variables $t=\sin\theta$ transforms the differential operator $\A_2$ as follows:
	\begin{equation*}
		\tilde{\A}_2
		= \Id + \frac{d^2}{d\theta^2}.
	\end{equation*}
	Therefore, there exist $c_1,c_2\in\R$ such that $g(\sin\theta)=c_1\cos\theta+c_2\sin\theta$ in $\D'(-\frac{\pi}{2},\frac{\pi}{2})$. Reversing the change of variables yields $g(t)=c_1\sqrt{1-t^2}+c_2t$ in $\D'(-1,1)$.
\end{proof}

The following lemma describes the action of $\A_1$ and $\A_2$ on Legendre polynomials.

\begin{lem}\label{Legendre_EV}
	For every $k\geq 2$,
	\begin{align}
		\frac{n-1}{(k-1)(k+n-1)}\A_1 P^n_k
		&= - \frac{k}{2k+n-2} P^{n+2}_{k-2} - \frac{k+n-2}{2k+n-2} P^{n+2}_k,
		\label{eq:Legendre_EV1}\\
		\frac{n-1}{(k-1)(k+n-1)}\A_2 P^n_k
		&= \frac{k(k+n-3)}{2k+n-2}P^{n+2}_{k-2} - \frac{(k+1)(k+n-2)}{2k+n-2}P^{n+2}_k.
		\label{eq:Legendre_EV2}
	\end{align}
\end{lem}
\begin{proof}
	We need the following two identities:
	\begin{equation}\label{eq:Legendre_EV:proof1}
		(2k+n-2)(n-1)P^n_k(t)
		= (k+n-2)(k+n-1)P^{n+2}_k(t) - (k-1)kP^{n+2}_{k-2}(t),
	\end{equation}
	\begin{equation}\label{eq:Legendre_EV:proof2}
		(k-1)P^n_{k-2}(t) - (2k+n-4)tP^n_{k-1}(t) + (k+n-3)P^n_k(t)
		= 0.
	\end{equation}	
	Both follow from \cref{eq:Legendre_deriv} by a simple inductive argument (see \cite[Section~3.3]{MR1412143}).
	By \cref{eq:Legendre_deriv} and \cref{eq:Legendre_EV:proof2},
	\begin{equation}\label{eq:Legendre_EV:proof3}
		(2k+n-2)(n-1)t\frac{d}{dt}P^n_k(t)
		= k(k+n-2)( (k-1)P^{n+2}_{k-2}(t)  + (k+n-1)P^{n+2}_k(t)).
	\end{equation}
	Combining \cref{eq:Legendre_ODE} with \cref{eq:Legendre_EV:proof1} and \cref{eq:Legendre_EV:proof3} yields
	\begin{equation}\label{eq:Legendre_EV:proof4}
			(2k+n-2)(n-1)(1-t^2)\frac{d^2}{dt^2}P^n_k(t)
			=  (k-1)k(k+n-2)(k+n-1)(P^{n+2}_{k-2}(t) - P^{n+2}_k(t)).
	\end{equation}
	 By \cref{eq:defi:EV} and a combination of identities \cref{eq:Legendre_EV:proof1}, \cref{eq:Legendre_EV:proof3}, and \cref{eq:Legendre_EV:proof4}, we obtain \cref{eq:Legendre_EV1} and \cref{eq:Legendre_EV2}.
\end{proof}

We use \cref{eq:Legendre_EV1} and \cref{eq:Legendre_EV2} to derive the following recurrence relation for multipliers.

\begin{lem}\label{multipliers_EV}
	Let $g$ be a locally integrable function on $(-1,1)$.
	\begin{enumerate}[label=\upshape(\roman*),topsep=0.5ex,itemsep=-0.5ex]
		\item If $(1-t^2)^{\frac{n-1}{2}}g'(t)\in\M(-1,1)$, then for all $k\geq 0$,
		\begin{equation}\label{eq:multipliers_EV_1}
			\frac{k-1}{2k+n}a^n_k[g] + \frac{k+n+1}{2k+n}a^n_{k+2}[g]
			= - \frac{1}{2\pi}a^{n+2}_k[\A_1g].
		\end{equation}
		\item If $(1-t^2)^{\frac{n+1}{2}}g''(t)\in\M(-1,1)$, then for all $k\geq 0$,
		\begin{equation}\label{eq:multipliers_EV_2}
			\frac{(k-1)(k+1)}{2k+n}a^n_k[g] - \frac{(k+n-1)(k+n+1)}{2k+n}a^n_{k+2}[g]
			= - \frac{1}{2\pi}a^{n+2}_k[\A_2 g].
		\end{equation}
	\end{enumerate}
\end{lem}
\begin{proof}
	In the following, we use that the family $(P^n_k)_{k=0}^\infty$ of Legendre polynomials is an orthogonal system with respect to the inner product $[\psi,g]_n=\int_{[-1,1]} \psi(t)g(t)(1-t^2)^{\frac{n-3}{2}}dt$ on $[-1,1]$. Moreover,  $[P^n_k,P^n_k]_n=\frac{\omega_n(k+n-2)}{\omega_{n-1}(2k+n-2)}\binom{k+n-2}{n-2}^{-1}$ (see, e.g., \cite[Section~3.3]{MR1412143}).	
	
	For the first part of the lemma, note that due to \cref{eq:beta_int_by_parts}, for every $\psi\in C^1[-1,1]$,
	\begin{equation}\label{eq:multipliers_EV_1:proof}
		[\psi,\A_1g]_{n+2}
		= [\A_1^\ast\psi,g]_n,
	\end{equation}
	where $\A_1^\ast$ denotes the differential operator
	\begin{equation*}
		\A_1^\ast
		= (1-(n+1)t^2)\Id + (1-t^2)t\frac{d}{dt}.
	\end{equation*}
	Clearly $\A_1^\ast$ increases the degree of a polynomial at most by two, so there exist $x_{k,j}\in\R$ such that for every $k\geq 0$,
	\begin{equation*}
		\A_1^\ast P^{n+2}_k
		= \sum_{j=0}^{k+2} x_{k,j} P^n_j.
	\end{equation*}
	Choosing $\psi=P^{n+2}_k$ in \cref{eq:multipliers_EV_1:proof} yields
	\begin{equation*}
		\frac{1}{\omega_{n+1}} a^{n+2}_k[\A_1 g]
		= [P^{n+2}_k, \A_1 g]_{n+2}
		= [\A_1^\ast P^{n+2}_k, g]_n
		= \sum_{j=0}^{k+2} x_{k,j} [P^n_j,g]_n
		= \frac{1}{\omega_{n-1}} \sum_{j=0}^{k+2} x_{k,j}a^n_j[g].
	\end{equation*}
	Hence, it only remains to determine the numbers $x_{k,j}$. By applying the identity above to $g=P^n_j$ for $0\leq j\leq k+2$, and employing \cref{eq:Legendre_EV1}, we obtain that
	\begin{equation*}
		x_{k,k}
		= - \frac{(k-1)(n-1)}{2k+n},
		\qquad
		x_{k,k+2}
		= - \frac{(k+n+1)(n-1)}{2k+n},
	\end{equation*}
	and $x_{k,j}=0$ for $j\notin\{k,k+2\}$, which proves \cref{eq:multipliers_EV_1}.
	
	For the second part of the proof, note that due to \cref{eq:beta_int_by_parts}, for every $\psi\in C^2[-1,1]$,
	\begin{equation}\label{eq:multipliers_EV_2:proof}
		[\psi,\A_2g]_{n+2}
		= [\A_2^\ast\psi,g]_n,
	\end{equation}
	where $\A_2^\ast$ denotes the differential operator
	\begin{equation*}
		\A_2^\ast
		= (1-t^2)^2\frac{d^2}{dt^2} + (n-1)((n+1)t^2-1)\Id - (2n+1)(1-t^2)t\frac{d}{dt}.
	\end{equation*}
	Clearly $\A_2^\ast$ increases the degree of a polynomial at most by two, so there exist $y_{k,j}\in\R$ such that for every $k\geq 0$,
	\begin{equation*}
		\A_2^\ast P^{n+2}_k
		= \sum_{j=0}^{k+2} y_{k,j} P^n_j.
	\end{equation*}
	Choosing $\psi=P^{n+2}_k$ in \cref{eq:multipliers_EV_2:proof} yields
	\begin{equation*}
		\frac{1}{\omega_{n+1}} a^{n+2}_k[\A_2 g]
		= [P^{n+2}_k, \A_2 g]_{n+2}
		= [\A_2^\ast P^{n+2}_k, g]_n
		= \sum_{j=0}^{k+2} y_{k,j} [P^n_j,g]_n
		= \frac{1}{\omega_{n-1}} \sum_{j=0}^{k+2} y_{k,j}a^n_j[g].
	\end{equation*}
	Hence, it only remains to determine the numbers $y_{k,j}$. By applying the identity above to $g=P^n_j$ for $0\leq j\leq k+2$, and employing \cref{eq:Legendre_EV2}, we obtain that
	\begin{equation*}
		y_{k,k}
		= - \frac{(k-1)(k+1)(n-1)}{2k+n},
		\qquad
		y_{k,k+2}
		= - \frac{(k+n-1)(k+n+1)(n-1)}{2k+n},
	\end{equation*}
	and $y_{k,j}=0$ for $j\notin\{k,k+2\}$, which proves \cref{eq:multipliers_EV_2}.
\end{proof}


We arrive at the following geometric inequality for convex bodies of revolution. This shows that equality is attained in \cref{eq:conv_ineq_k=2} only by line segments, which completes the proof of \cref{fixed_points_body_of_rev}.

\begin{thm}\label{convex_body_2nd_muliplier}
	Let $L\in\K^n$ be a convex body of revolution. Then
	\begin{equation}\label{eq:convex_body_2nd_multiplier}
		- \frac{1}{n-1}
		\leq \frac{a^n_2[\square_n h_L]}{a^n_0[\square_n h_L]}
		\leq \frac{1}{(n-1)^2}
	\end{equation}
	with equality in the left hand inequality if and only if $L$ is a line segment and equality in the right hand inequality if and only if $L$ is an $(n-1)$-dimensional disk.
\end{thm}
\begin{proof}
	As an instance of \cref{eq:multipliers_EV_1},
	\begin{equation*}
		a^n_2[\square_nh_L] + \frac{1}{n-1}a^n_0[\square_nh_L]
		= \frac{n}{2\pi(n-1)}a^{n+2}_0[\A_1 \overline{h_L}].
	\end{equation*}
	Due to \cref{support_fct_EV}, we have that $\A_1 \overline{h_L}\geq 0$, which proves the first inequality in \cref{eq:convex_body_2nd_multiplier}. Moreover, equality holds precisely when $\A_1 \overline{h_L}=0$. According to \cref{kernel_EV}~\ref{kernel_EV1}, this is the case if and only if $\overline{h_L}(t)=c_1\abs{t}+c_2 t$ for some $c_1,c_2\in\R$, which means that $L$ is a line segment.
	
	As an instance of \cref{eq:multipliers_EV_2},
	\begin{equation*}
		a^n_2[\square_nh_L] - \frac{1}{(n-1)^2}a^n_0[\square_nh_L]
		= \frac{n}{2\pi(n-1)^2}a^{n+2}_0[\A_2 \overline{h_L}].
	\end{equation*}
	Due to \cref{support_fct_EV}, we have that $\A_2 \overline{h_L}\geq 0$, which proves the second inequality in \cref{eq:convex_body_2nd_multiplier}. Moreover, equality holds precisely when $\A_2 \overline{h_L}=0$. According to \cref{kernel_EV}~\ref{kernel_EV2}, this is the case if and only if $\overline{h_L}(t)=c_1\sqrt{1-t^2}+c_2 t$ for some $c_1,c_2\in\R$, which means that $L$ is an $(n-1)$-dimensional disk.
\end{proof}

\subsection{Even Minkowski Valuations}
\label{sec:even_i<n-1}

This section is dedicated to even Minkowski valuations $\Phi_i\in\MVal_i$ of degree $1<i\leq n-1$. In the previous subsection, we have shown that if $\Phi_i$ is generated by an origin-symmetric convex body of revolution, then condition \ref{fixed_points_Phi^2:inj} is fulfilled, unless $\Phi_i$ is a multiple of the projection body map (see Theorems~\ref{conv_ineq} and \ref{convex_body_2nd_muliplier}).

In general, the generating function of an even Minkowski valuation does not need to be a support function. In this broader setting, we prove a weaker condition than \ref{fixed_points_Phi^2:inj}, which we use to obtain information about the fixed points of the map $\Phi_i$ itself as opposed to $\Phi_i^2$. To that end, we require the following version of \cref{fixed_points_Phi^2}, which can be obtained from a minor modification of its proof, as was observed in \cite{MR4316669}.

%
%

\begin{thm}[{\cite{MR4316669}}]\label{fixed_points_Phi^1}
	Let $1<i\leq n-1$ and $\Phi_i\in\MVal_i$ with generating function $f$ satisfying conditions \ref{fixed_points_Phi^2:diff}, \ref{fixed_points_Phi^2:surj}, and
	\begin{enumerate}[label=\upshape(C\arabic{*}'),topsep=0.5ex,itemsep=-0.5ex,leftmargin=12mm]
		\setcounter{enumi}{2}
		\item \label{fixed_points_Phi^1:inj} for all $k\geq 2$,
		\begin{equation*}\label{eq:fixed_points_Phi^1:inj}
			\frac{a^n_k[\square_nf]}{a^n_0[\square_nf]}
			< \frac{1}{i}.
		\end{equation*}
	\end{enumerate}
	Then there exists a $C^2$ neighborhood of $B^n$ where the only fixed points of $\Phi_i$ are Euclidean balls.
\end{thm}

Again, Corollaries~\ref{w_mon=>bounded_C^2} and \ref{w_mon=>decay} show that conditions \ref{fixed_points_Phi^2:diff} and \ref{fixed_points_Phi^2:surj} are fulfilled in the weakly monotone case. Hence we obtain the following.

\begin{thm}\label{w_mon=>fixed_points_Phi^1}
	Let $1< i\leq n-1$ and $\Phi_i\in\MVal_i$ be weakly monotone with generating function $f$ satisfying condition \ref{fixed_points_Phi^1:inj}. Then there exists a $C^2$ neighborhood of $B^n$ where the only fixed points of $\Phi_i$ are Euclidean balls.
\end{thm}

The main result of this section will be that if $\Phi_i\in\MVal_i$ is even, then its generating function satisfies condition \ref{fixed_points_Phi^1:inj}. We require the following lemma, which is a consequence of a classical result by Firey \cite{MR271838}. We call a convex body of revolution \emph{smooth} if it has a $C^2(\S^{n-1})$ support function and $\A_2 \overline{h_K}>0$ on $[-1,1]$.

\begin{lem}\label{area_meas_zonal}
	Let $1\leq i< n-1$ and let $\phi\in C(\S^{n-1})$ be zonal and centered. Then $\phi$ is the density of the $i$-th area measure of a smooth convex body of revolution if an only if for all $t\in (-1,1)$,
	\begin{equation*}\label{eq:area_meas_zonal}
		\bar{\phi}(t)
		> \frac{n-1-i}{n-1}\A_1\big(\overline{\phi\ast\breve{g}_n}\big)(t)
		> 0.
	\end{equation*}
\end{lem}
\begin{proof}
	It was proved in \cite{MR271838} that $\phi$ is the density of the $i$-th area measure of a smooth convex body if and only if for all $t\in (-1,1)$,
	\begin{equation*}
		\bar{\phi}(t)
		> (n-1-i) (1-t^2)^{-\frac{n-1}{2}}\int_{(t,1)} \bar{\phi}(s) s (1-s^2)^{\frac{n-3}{2}}ds
		> 0.
	\end{equation*}
	Therefore it only remains to show that for all $t\in (-1,1)$,
	\begin{equation}\label{eq:area_meas_zonal:proof}
		\int_{(t,1)} \bar{\phi}(s) s (1-s^2)^{\frac{n-3}{2}}ds
		= \frac{1}{n-1}(1-t^2)^{\frac{n-1}{2}}\A_1\big(\overline{\phi\ast\breve{g}_n}\big)(t).
	\end{equation}
	We have seen in \cref{exl:Berg_fct_convol_trafo} that the convolution transform $\T_{\breve{g}_n}$ is a bounded operator from $C(\S^{n-1})$ to $C^1(\S^{n-1})$, so both sides of \cref{eq:area_meas_zonal:proof} depend continuously on $\phi\in C(\S^{n-1})$ with respect to uniform convergence. Therefore it suffices to show \cref{eq:area_meas_zonal:proof} only for smooth $\phi$.
	
	To that end, let $\zeta=\phi\ast\breve{g}_n\in C^\infty(\S^{n-1})$ and observe that according to \cref{eq:zonal_smooth_deriv3}, 
	\begin{equation*}
		\bar{\phi}(s)
		= \overline{\square_n\zeta}(s)
		= \frac{1}{n-1}\overline{\Delta_{\S}\zeta}(s) + \bar{\zeta}(s) 
		= \frac{1}{n-1}(1-s^2)\bar{\zeta}''(s) + \bar{\zeta}(s)-s\bar{\zeta}'(s).
	\end{equation*}
	A direct computation yields
	\begin{equation*}
		\bar{\phi}(s) s (1-s^2)^{\frac{n-3}{2}}
		= - \frac{1}{n-1}\frac{d}{ds}\left((1-s^2)^{\frac{n-1}{2}}\A_1\bar{\zeta}(s)\right).
	\end{equation*}
	Hence, we obtain \cref{eq:area_meas_zonal:proof}, which completes the proof.
\end{proof}

Next, we prove the following two technical lemmas.
For smooth functions $\psi\in C^\infty[-1,1]$, we define $\overline{\square}_n\psi=\overline{\square_n\psi(\langle\e,\cdot\rangle)}$. 
Note that $\overline{\square}_n\psi(t)=\frac{1}{n-1}(1-t^2)\psi''(t)+\psi(t)-t\psi'(t)$ due to \cref{eq:zonal_smooth_deriv3}.

\begin{lem}\label{EV1_max}
	For every $\psi\in C^\infty[-1,1]$,
	\begin{equation}\label{eq:EV1_max}
		\max_{[-1,1]} \A_1 \psi
		\leq \max_{[-1,1]} \overline{\square}_n \psi.
	\end{equation}
\end{lem}
\begin{proof}
	Let $t_0\in[-1,1]$ be a maximum point of $\A_1 \psi$. We will show that
	\begin{equation}\label{eq:EV1_max:proof}
		(1-t_0^2)\psi''(t_0)\geq 0.
	\end{equation}
	If $t_0=\pm 1$, then clearly we have \cref{eq:EV1_max:proof}. If $t_0\in (-1,1)$, then 
	\begin{equation*}
		-t_0 \psi''(t_0)=(\A_1\psi)'(t_0)=0
		\qquad\text{and}\qquad
		-\psi''(t_0)-t_0 \psi'''(t_0)=(\A_1\psi)''(t_0)\leq 0,
	\end{equation*}
	which implies that $t_0=0$ or that $\psi''(t_0)=0$. In the latter case, we obtain \cref{eq:EV1_max:proof} again. In the case where $t_0=0$, we obtain that $-\psi''(t_0)\leq 0$, which also yields \cref{eq:EV1_max:proof}. Therefore
	\begin{equation*}
		\A_1 \psi(t_0)
		= \overline{\square}_n \psi(t_0) - \tfrac{1}{n-1}(1-t_0^2)\psi''(t_0)
		\leq \overline{\square}_n \psi(t_0),
	\end{equation*}
	which proves \cref{eq:EV1_max}.
\end{proof}

\begin{lem}\label{EV1_min}
	For every $k\geq 2$,
	\begin{equation}\label{eq:EV1_min}
		\min_{[-1,1]} \A_1 P^n_k
		= \A_1P^n_k(1)
		= - \frac{(k-1)(k+n-1)}{n-1}.
	\end{equation}	
\end{lem}
\begin{proof}
	According to \cref{eq:Legendre_EV1}, the function $a^n_k[\square_n]^{-1}\A_1P^n_k$ is a convex combination of the two Legendre polynomials $P^{n+2}_{k-2}$ and $P^{n+2}_k$. They both have $1$ as their maximum value on $[-1,1]$ and they both attain it at $t_0=1$. Therefore, this must also be the case for $a^n_k[\square_n]^{-1}\A_1P^n_k$, which proves \cref{eq:EV1_min}.
\end{proof}

We now define a family of polynomials that turns out to be instrumental in the following.

\begin{defi}
	For $1\leq i\leq n-1$, $k\geq 0$ and $k\neq 1$, we define
	\begin{equation}
		Q^n_{k,i}
		= P^n_k + \frac{n-1-i}{(k-1)(k+n-1)}\A_1P^n_k.
	\end{equation}
\end{defi}

Observe that for $i=n-1$, the polynomial $Q^n_{k,n-1}$ is the classical Legendre polynomial $P^n_k$. Denote the extrema of $Q^n_{k,i}$ on the interval $[-1,1]$ by
\begin{equation*}
	m^n_{k,i}
	= \min_{[-1,1]}Q^n_{k,i}
	\qquad\text{and}\qquad
	M^n_{k,i}
	= \max_{[-1,1]}Q^n_{k,i}.
\end{equation*}
The following lemma about the minima $m^n_{k,i}$ is why we require $k$ to be even.

\begin{lem}\label{min_Q^n_ki_monotone}
	Let $k\geq 2$ be even. Then the sequence $(m^n_{k,i})_{i=1}^{n-1}$ is strictly increasing, that is,
	\begin{equation}\label{eq:min_Q^n_ki_monotone}
		m^n_{k,1}
		< m^n_{k,2}
		< \cdots
		< m^n_{k,n-1}.
	\end{equation}
\end{lem}
\begin{proof}
	For fixed even $k\geq 2$, define a family $(\eta_t)_{t\in[-1,1]}$ of affine functions by
	\begin{equation*}
		\eta_t(s)
		= P^n_k(t) + s \A_1P^n_k(t)
	\end{equation*}
	and observe that it suffices to show that the function $\eta$ defined by
	\begin{equation*}
		\eta(s)
		= \min_{t\in[-1,1]} \eta_t(s)
		= \min_{[-1,1]} \left\{ P^n_k + s \A_1P^n_k\right\}
	\end{equation*}
	is strictly decreasing on $[0,\infty)$.
	
	To that end, note that as the point-wise minimum of a family of affine functions, $\eta$ is a concave function. Next, note that since $P^n_k$ is an even Legendre polynomial, it is minimized in the interior of $[-1,1]$, that is, there exists $t_0\in (-1,1)$ such that
	\begin{equation*}
		\eta(0)
		= \min_{[-1,1]}P^n_k
		= P^n_k(t_0)
		< 0.
	\end{equation*}
	Moreover, $\frac{d}{dt}P^n_k(t_0)=0$, so for every $s>0$,
	\begin{equation*}
		\eta(s)
		\leq \eta_{t_0}(s)
		= P^n_k(t_0) + s\A_1P^n_k(t_0)
		= (1+s)P^n_k(t_0)
		< P^n_k(t_0)
		= \eta(0).
	\end{equation*}
	Since $\eta$ is concave, this implies that $\eta$ is strictly decreasing on $[0,\infty)$, which completes the proof.
\end{proof}

The following two propositions are an extension of \cite[Proposition~5.4]{MR4316669}.

\begin{prop}\label{interval_i-th_area_meas}
	For $1\leq i\leq n-1$ and $k\geq 2$, denote by $J^n_{k,i}$ the set of all $\lambda\in\R$ for which $1+\lambda P^n_k(\langle\e,\cdot\rangle)$ is the density of the $i$-th area measure of a smooth convex body of revolution. If $1\leq i<n-1$ and $k\geq 2$ is even, then
	\begin{equation}\label{eq:interval_i-th_area_meas}
		\left( - \frac{i}{(n-1)M^{n}_{k,i}},- \frac{i}{(n-1)m^{n}_{k,i}} \right)
		\subseteq J^n_{k,i}
		\subseteq \left[ - \frac{i}{(n-1)M^{n}_{k,i}},- \frac{i}{(n-1)m^{n}_{k,i}} \right].
	\end{equation}
	Moreover, if $i=n-1$, then the interval on the right hand side of \cref{eq:interval_i-th_area_meas} is precisely the set of all $\lambda\in\R$ for which $1+\lambda P^n_k(\langle\e,\cdot\rangle)$ is the density of the surface area measure of a convex body of revolution.
\end{prop}
\begin{proof}
	To simplify notation, all minima and maxima in this proof refer to the interval $[-1,1]$.
	\cref{area_meas_zonal} shows that $\lambda\in J^n_{k,i}$ if and only if
	\begin{equation}\label{eq:interval_i-th_area_meas:proof1}
		1+\lambda P^n_k(t)
		> \frac{n-1-i}{n-1} - \lambda\frac{n-1-i}{(k-1)(k+n-1)}\A_1P^n_k(t)
		> 0
	\end{equation}
	for all $t\in (-1,1)$. An easy rearrangement of \cref{eq:interval_i-th_area_meas:proof1} implies the right hand set inclusion in \cref{eq:interval_i-th_area_meas}.
	For the other set inclusion, let $\lambda\in\R$ and suppose that
	\begin{equation}\label{eq:interval_i-th_area_meas:proof2}
		- \frac{i}{(n-1)M^n_{k,i}}
		< \lambda
		< - \frac{i}{(n-1)m^n_{k,i}}.
	\end{equation}
	Due to \cref{eq:EV1_min}, we have that
	\begin{equation}\label{eq:interval_i-th_area_meas:proof3}
		\frac{(k-1)(k+n-1)}{(n-1)\min \A_1P^n_k}
		= -1
		= - \frac{i}{(n-1)Q^n_{k,i}(1)}
		\leq - \frac{i}{(n-1)M^n_{k,i}}
		< \lambda.
	\end{equation}
	Moreover, \cref{eq:min_Q^n_ki_monotone} combined with \cref{eq:EV1_max} applied to $P^n_k$ yields
	\begin{equation}\label{eq:interval_i-th_area_meas:proof4}
		\lambda
		< - \frac{i}{(n-1)m^n_{k,i}}
		\leq - \frac{1}{m^n_{k,i}}
		< - \frac{1}{m^n_{k,n-1}}
		= \frac{(k-1)(k+n-1)}{(n-1)\max \overline{\square}_nP^n_k}
		\leq \frac{(k-1)(k+n-1)}{(n-1)\max \A_1P^n_k}.
	\end{equation}
	Finally, observe that \cref{eq:interval_i-th_area_meas:proof2}, \cref{eq:interval_i-th_area_meas:proof3}, and \cref{eq:interval_i-th_area_meas:proof4} jointly imply \cref{eq:interval_i-th_area_meas:proof1}, thus $\lambda\in J^n_{k,i}$. This shows the left hand set inclusion in \cref{eq:interval_i-th_area_meas}.
	
	For the second part of the proposition, observe that $\lambda$ lies in the interval on the right hand side of \cref{eq:interval_i-th_area_meas} precisely when $1+\lambda P^n_k(\langle\e,\cdot\rangle)\geq 0$. According to Minkowski's existence theorem (see, e.g., \cite[p.~455]{MR3155183}), this is the case if and only if $1+\lambda P^n_k(\langle\e,\cdot\rangle)$ is the density of the surface area measure of a convex body.
\end{proof}

\begin{prop}\label{interval_support_fct}
	Let $k\geq 2$ be even and $I^n_{k}$ denote the set of all $\lambda\in\R$ for which $1+\lambda P^n_k(\langle\e,\cdot\rangle)$ is the support function of a convex body of revolution $K_\lambda$. Then
	\begin{equation}\label{eq:interval_support_fct}
		I^n_{k}
		= \left[ - \frac{1}{(k-1)(k+n-1)m^{n}_{k,1}},- \frac{1}{(k-1)(k+n-1)M^{n}_{k,1}} \right].
	\end{equation}
\end{prop}
\begin{proof}
	Denote the interval on the right hand side of \cref{eq:interval_support_fct} by $\tilde{I}^n_k$.
	Since the space of support functions is a closed convex cone of $C(\S^{n-1})$, the set $I^n_k$ must be a closed interval. Recall that $S_1(K,\cdot)=\square_n h(K,\cdot)$ for every convex body $K\in\K^n$. Hence, \cref{eq:interval_i-th_area_meas} shows that
	\begin{equation*}
		I^n_k
		\supseteq a^n_k[\square_n]^{-1}\cl(J^n_{k,1})
		= \tilde{I}^n_k,
	\end{equation*}
	where $\cl$ denotes the closure.
	For the converse set inclusion, let $\lambda\in I^n_k$. Then $D^2h_{K_\lambda}(u)=P_{u^\perp} + \lambda D^2\breve{P}^n_k(u)$ is positive semidefinite for all $u\in\S^{n-1}$. This implies that for every $\varepsilon>0$, the matrix $P_{u^\perp} + (1-\varepsilon)\lambda D^2\breve{P}^n_k(u)$ is positive definite for all $u\in\S^{n-1}$. Therefore, $1+(1-\varepsilon)\lambda\breve{P}^n_k$ is the support function of a convex body of revolution which is of class $C^\infty_+$, and thus, strictly convex (see, e.g., \cite[Section~2.5]{MR3155183}). Hence, \cref{eq:interval_i-th_area_meas} implies that $(1-\varepsilon)\lambda\in\tilde{I}^n_k$ for all $\varepsilon>0$, and thus, $\lambda\in\tilde{I}^n_k$.
\end{proof}


We are now in a position to prove the main result of this section.

\begin{thm}\label{even=>multiplier_ineq}
	Let $1< i\leq n-1$ and $\Phi_i\in\MVal_i$ be non-trivial. Then its generating function $f$ satisfies for all even $k\geq 2$,
	\begin{equation*}
		\frac{a^n_k[\square_n f]}{a^n_0[\square_n f]}
		< \frac{1}{i}.
	\end{equation*}
\end{thm}
\begin{proof}
	First, observe that for every convex body $K\in\K^n$,
	\begin{equation*}
		S_i(K,\cdot)\ast\square_nf
		= \square_n(S_i(K,\cdot)\ast f)
		= \square_nh(\Phi_iK,\cdot)
		= S_1(\Phi_iK,\cdot),
	\end{equation*}
	thus the convolution transform $\T_{\square_nf}$ maps $i$-th order area measures to first order area measures. Moreover, for every $\lambda\in\R$,
	\begin{equation*}
		(1+\lambda P^n_k(\langle\e,\cdot\rangle))\ast \square_nf
		= a^n_0[\square_nf] + \lambda a^n_k[\square_nf] P^n_k(\langle\e,\cdot\rangle).
	\end{equation*}
	Hence, we obtain that
	\begin{equation*}
		\frac{a^n_k[\square_n f]}{a^n_0[\square_n f]} J^n_{k,i}
		\subseteq a^n_k[\square_n]I^n_{k}.
	\end{equation*}
	The descriptions of the intervals $J^n_{k,i}$ and $I^n_{k}$ given in \cref{eq:interval_i-th_area_meas} and \cref{eq:interval_support_fct} imply that
	\begin{equation*}
		\frac{a^n_k[\square_n f]}{a^n_0[\square_n f]}
		\leq \frac{1}{i} \frac{m^n_{k,i}}{m^n_{k,1}}
		< \frac{1}{i},
	\end{equation*}
	where the strict inequality is due to \cref{eq:min_Q^n_ki_monotone}.
\end{proof}

Combining \cref{even=>multiplier_ineq} with \cref{w_mon=>fixed_points_Phi^1}, we obtain the following.

\begin{cor}\label{w_mon_even=>fixed_points}
	Let $1< i\leq n-1$ and $\Phi_i\in\MVal_i$ be weakly monotone and even. Then there exists a $C^2$ neighborhood of $B^n$ where the only fixed points of $\Phi_i$ are Euclidean balls.
\end{cor}

\begin{rem}
	Computational simulations suggest that for every $1\leq i\leq n-1$ and $k\geq 2$, the maximum of $Q^n_{k,i}$ on $[-1,1]$ is attained in $t=1$, that is,
	\begin{equation}\label{eq:max_Q^n_k_conj}
		M^n_{k,i}
		= Q^n_{k,i}(1)
		= \frac{i}{n-1}.
	\end{equation}
	As an immediate consequence, the intervals in \cref{eq:interval_i-th_area_meas} could be simplified.
	
	Computational simulations also suggest that for $1\leq i\leq n-1$ and for every even $k\geq 4$,
	\begin{equation}\label{eq:min_Q^n_ki_conj}
		- \frac{1}{n-1}
		< m^n_{k,i}.
	\end{equation}
	If both \cref{eq:max_Q^n_k_conj} and \cref{eq:min_Q^n_ki_conj} were shown to be true, then the argument in the proof of \cref{even=>multiplier_ineq} would immediately imply that whenever $1<i\leq n-1$ and $\Phi_i\in\MVal_i$ is non-trivial with generating function $f$, then for all even $k\geq 4$,
	\begin{equation*}
		- \frac{1}{i}
		< \frac{a^n_k[\square_n f]}{a^n_0[\square_nf]}.
	\end{equation*}
\end{rem}

\appendix

\section{Appendix}
\label{sec:omitted}

\begin{proof}[Proof of \cref{beta_int_by_parts}]
	We may assume that $(1-t^2)^{\frac{\beta}{2}}g'(t)$ is a positive measure: all statements of the lemma follow from this case by the Jordan decomposition theorem and linearity. Thus $g'$ itself is a locally finite positive measure on $(-1,1)$, so there exists some constant $c\in\R$ such that for almost all $t\in (-1,1)$,
	\begin{equation*}
		g(t)
		= \left\{\begin{array}{ll}
			c-g'((t,0]),	&t<0,	\\
			c+g'((0,t]),	&t\geq 0.
		\end{array}\right.
	\end{equation*}
	We may assume that $c=g(0)=0$. Since $g$ is an increasing function, we have that $g\leq 0$ on $(-1,0]$ and $g\geq 0$ on $[0,1)$.
	
	For $0<a<1$, Lebesgue-Stieltjes integration by parts yields
	\begin{align*}
		\beta \int_{(-a,a]} t(1-t^2)^{\frac{\beta-2}{2}}g(t)dt
		&= \int_{(-a,a]} (1-t^2)^{\frac{\beta}{2}} g'(dt) - g'((-a,a])(1-a^2)^{\frac{\beta}{2}} \\
		&\leq \int_{(-a,a]} (1-t^2)^{\frac{\beta}{2}} g'(dt).
	\end{align*}
	By passing to the limit $a\to 1^-$ and applying the monotone convergence theorem, we obtain that $(1-t^2)^{\frac{\beta-2}{2}}g(t)$ is integrable on $(-1,1)$.
	
	For the second part of the lemma, note that since $g$ is increasing,
	\begin{equation*}
		g(a)(1-a^2)^{\frac{\beta}{2}}
		= \beta g(a)\int_{(a,1)} t(1-t^2)^{\frac{\beta-2}{2}}dt
		\leq \beta\int_{(a,1)} g(t)t(1-t^2)^{\frac{\beta-2}{2}}dt
	\end{equation*}
	and the right hand side tends to zero as $a$ tends to $1$. An analogous argument applies to $-a$, thus
	\begin{equation*}
		\lim_{a\to 1^-} g(a)(1-a^2)^{\frac{\beta}{2}}
		= \lim_{a\to 1^-} g(-a)(1-a^2)^{\frac{\beta}{2}}
		= 0.
	\end{equation*}
	Suppose now that $\psi$ is as stated above. For $0<a<1$, Lebesgue-Stieltjes integration by parts yields
	\begin{equation*}
		\int_{(-a,a]} \psi(t)g'(dt) + \int_{(-a,a]} \psi'(t)g(t)dt
		= (g(a)-g(-a))(\psi(a)-\psi(-a)).
	\end{equation*}
	Due to our assumptions on $\psi$, the right hand hand side tends to zero as $a$ tends to $1$. Thus, by passing to the limit $a\to1^-$ and applying the dominated convergence theorem, we obtain \cref{eq:beta_int_by_parts}.
\end{proof}

\begin{proof}[Proof of \cref{polar_C^k_estimate}]
	First, fix $v$, $\alpha$ and $\beta$ and observe that it suffices to find a family of bounded linear operators $D_k: C^k(\S^{n-1}\backslash\{\pm v\})\to C(\S^{n-1}\backslash\{\pm v\})$ such that for every $\phi\in C^\infty(\S^{n-1}\backslash\{\pm v\})$ and $w\in v^\perp$,
	\begin{equation*}
		\frac{d^k}{dt^k} \J_v[\langle\cdot,w\rangle^\alpha (1-\langle\cdot,v\rangle^2)^{\frac{\beta}{2}}\phi](t)
		= (1-t^2)^{-k} \J_v[\langle\cdot,w\rangle^\alpha (1-\langle\cdot,v\rangle^2)^{\frac{\beta}{2}}D_k\phi](t).
	\end{equation*}	
	We will construct this family $D_k$ inductively, starting with $D_0=\Id$.
	
	For the induction step, define a first order differential operator $\tilde{D}_k$ by
	\begin{equation*}
		\tilde{D}_k \phi (u)
		= \langle \nabla_{\S} \phi (u), P_{u^\perp} v \rangle - 2\left(\tfrac{n-3+\alpha+\beta}{2}-(k-1)\right)\langle u,v\rangle\phi(u).
	\end{equation*}
	A straightforward computation using spherical cylinder coordinates shows that
	\begin{align*}
		&\frac{d}{dt} \left((1-t^2)^{-(k-1)}\J_v[\langle\cdot,w\rangle^\alpha (1-\langle\cdot,v\rangle^2)^{\frac{\beta}{2}}\phi](t)\right) \\
		&\qquad= \frac{d}{dt}\left( (1-t^2)^{\frac{n-3+\alpha+\beta}{2}-(k-1)}\int_{\S^{n-1}\cap v} \phi(tv+\sqrt{1-t^2}u) du \right) \\
		&\qquad= (1-t^2)^{-k} \J_v[\langle\cdot,w\rangle^\alpha (1-\langle\cdot,v\rangle^2)^{\frac{\beta}{2}}\tilde{D}_k\phi](t),
	\end{align*}
	thus we see that the operators $D_k=\tilde{D}_k\tilde{D}_{k-1}\cdots\tilde{D}_1$ have the desired property. Since every $\tilde{D}_j$ is a bounded linear operator from $C^j(\S^{n-1}\backslash\{\pm v\})$ to $C^{j-1}(\S^{n-1}\backslash\{\pm v\})$, it follows by induction that every $D_k$ is a bounded linear operator from $C^k(\S^{n-1}\backslash\{\pm v\})$ to $C(\S^{n-1}\backslash\{\pm v\})$.
\end{proof}

\begin{proof}[Proof of \cref{w_pos_iff}]
	Suppose that $\nu\in C^{-\infty}(\mathbb{S}^{n-1})$ is weakly positive, that is, $\nu=\mu+y$ for some positive measure $\mu$ and some linear function $y$. Then for every positive centered $\phi\in C^\infty(\mathbb{S}^{n-1})$, we have that
	$\left< \phi, \nu \right>_{C^{-\infty}}= \left< \phi, \mu +y\right>_{C^{-\infty}}\geq 0.$
	
	Conversely, suppose that $\nu\in C^{-\infty}(\mathbb{S}^{n-1})$ is not weakly positive. Observe that the set of weakly positive distributions is a closed convex cone of $C^{-\infty}(\S^{n-1})$. Due to the Hahn-Banach separation theorem there exists some $\phi\in C^{\infty}(\S^{n-1})$ such that
	\begin{equation*}
		\left< \phi , \nu \right>_{C^{-\infty}}
		< \left< \phi, \mu + y \right>_{C^{-\infty}}
	\end{equation*}
	for every positive measure $\mu$ and linear function $y$. By fixing $\mu=0$ and varying $y\in\mathcal{H}^n_1$, we see that $\phi$ is centered. By fixing $y=0$ and varying $\mu\in\M_+(\S^{n-1})$, we see that $\phi\geq 0$. Finally, by choosing $\mu=y=0$, we see that $\left< \phi,\nu \right>_{C^{-\infty}}<0$.
\end{proof}

\section*{Acknowledgments}

The second author was supported by the Austrian Science Fund (FWF), Project numbers: P31448-N35 and ESP~236 ESPRIT-Programm.

\begingroup
\let\itshape\upshape

\bibliographystyle{abbrv}
\bibliography{references}{}

\endgroup

\hspace{5mm}

\noindent Leo Brauner

\noindent TU Wien

\noindent braunerleo@gmail.com

\hspace{5mm}

\noindent Oscar Ortega-Moreno

\noindent TU Wien

\noindent oscarortem@gmail.com

\end{document}